\documentclass[reqno, 12pt]{amsart}

\usepackage{cite}
\usepackage{graphicx}
\usepackage[a4paper,centering, scale=0.8]{geometry}
\date{}

\allowdisplaybreaks[4]
\theoremstyle{plain}
\newtheorem{thm}{Theorem}[section]

\newtheorem{lem}[thm]{Lemma}

\theoremstyle{remark}

\theoremstyle{definition}
\newtheorem{definition}{Definition}

\usepackage[colorlinks,linkcolor=blue,urlcolor=red,linktocpage]{hyperref}

\numberwithin{equation}{section}

\begin{document}

\title{Petty projection inequality on the sphere and on  the hyperbolic space}


\author[Y. Lin]{Youjiang Lin}
\address[Y. Lin]{School of Mathematics and Statistics, Chongqing Technology and Business University, Chongqing 400067, PR China; Dipartimento di Matematica e Informatica ``U. Dini", Universit\'a di Firenze, Viale Morgagni 67/a, 50134 Firenze, Italy}
 \email{\href{mailto: YOUJIANG LIN
<yjl@ctbu.edu.cn; youjiang.lin@unifi.it>}{yjl@ctbu.edu.cn; youjiang.lin@unifi.it}}

\author[Y. Wu]{Yuchi Wu}
\address[Y. Wu]{School of Mathematical Sciences, Key Laboratory of MEA(Ministry of Education) \& Shanghai Key Laboratory of PMMP, East China Normal University, Shanghai 200241, China}
 \email{\href{mailto: Yuchi Wu
<ycwu@math.ecnu.edu.cn>}{ycwu@math.ecnu.edu.cn}}

\begin{abstract}
Using gnomonic projection and Poincar\'e model, we first define the spherical projection body and hyperbolic projection body in spherical space $\mathbb{S}^n$ and hyperbolic space $\mathbb{H}^n$, then define the spherical Steiner symmetrization and hyperbolic Steiner symmetrization, finally prove the spherical projection inequality and hyperbolic projection inequality.
\end{abstract}

\footnote{Research of the first named author is supported by NSFC 11971080 and NSFC 12371137 and Jiangxi Provincial Natural Science Foundation 20232BAB201005. Research of the second author is supported by Science and Technology Commission of Shanghai Municipality 22DZ2229014.}
\footnote{There is no conflict of interest and no assocaited data in this manuscript.}
\maketitle
\section{Introduction}
%
%
%
%
%

In the Brunn-Minkowski theory of Euclidean space $\mathbb{R}^n$, the two classical inequalities which connect the volume of a convex body with that of its polar projection body are the Petty and Zhang projection inequalities, see e.g. \cite{MR0840399,Zhang91}. The Petty projection inequality shows that among all convex bodies with the same volume, the ellipsoids have the largest volume of polar projection body. The Zhang projection inequality shows that the simplices minimise the  volume of polar projection body. Petty projection inequalities have been extended to the $L_p$ Petty projection inequalities and Orlicz projection inequalities, see e.g. \cite{LYZ00,MR2563216}. Moreover, the functional versions of Petty projection inequality---affine P\'olya-Szeg\"o inequality and affine Sobolev inequality have been largely studies, see e.g. \cite{CLYZ09,LYZ02,Lin17,Zhang99}.

Recently some researches on isoperimetry in Euclidean space have been extended to spherical or hyperbolic space, see e.g. \cite{Beckner93,BDF16,BS19,DKY18,Makowski,Yaskin06,BW16,BHPS23,GHS03,HP21,WX14}. F. Besau and E.M. Werner \cite{BW16} introduced the spherical convex floating body for a convex body on the Euclidean unit sphere and define a new spherical area measure---the floating area. We are convinced that the
floating area will become a powerful tool in spherical convex geometry. F. Besau, T. Hack, P. Pivovarov and F.E. Schuster \cite{BHPS23} introduced the spherical centroid body of a centrally-symmetric convex body in the Euclidean unit sphere, studied a number of basic properties of spherical centroid bodies and proved a spherical analogue of the classical polar Busemann-Petty centroid inequality. F. Gao, D. Hug and R. Schneider \cite{GHS03} proved the Urysohn inequality and the Blaschke-Santal\'o inequality in the spherical and hyperbolic space. Later, T. Hack and P. Pivovarov \cite{HP21} proved a randomized version of the spherical and hyperbolic Urysohn-type inequalities. G. Wang and C. Xia \cite{WX14} solved various isoperimetric problems for the quermassintegrals and the curvature integrals in the hyperbolic space $\mathbb{H}^n$ and established quite strong Alexandrov-Fenchel type inequalities.

In comparison with the Euclidean case there are very few results and techniques available in spherical and hyperbolic space. In \cite{GHS03} and \cite{HP21}, the authors
use the two-point symmetrization procedure together with rearrangement inequalities  in order to prove their result. In \cite{BHPS23}, the authors used a probabilistic approach to prove the spherical centroid inequality. In \cite{WX14}, the authors used quermassintegral-preserving curvature flows approach to solve the isoperimetric type problems in the hyperbolic space. It is well-known that Steiner symmetrization is a fundamental tool for attacking problems regarding isoperimetry and related geometric inequalities in Euclidean space $\mathbb{R}^n$, see e.g. \cite{BGG17,CCF05,CF06}. In this paper, we define the spherical and hyperbolic Steiner symmetrizations on $\mathbb{S}^n$ and $\mathbb{H}^n$ which preserve the property of star bodies,  the volume invariance after a Steiner symmetrization and convergence of iterative Steiner symmetrizations. Using the spherical and hyperbolic Steiner symmetrizations, we prove the spherical projection inequality and hyperbolic projection inequality, respectively.

In this paper, first we study the spherical projection body in spherical space, which is a natural analog of Petty projection body in Euclidean space. Using the gnomonic projection from $\mathbb{S}^{n-1}$ onto $\mathbb{R}^n$ and the definition of classical Euclidean Petty projection body, we define the spherical Petty projection body. Using the monotonic increasing property of the measures of the polar bodies of spherical projection bodies after performing a spherical Steiner symmetrization and the continuity of spherical projection operator with respect to spherical Hausdorff distance, we proved the spherical Petty projection inequality.
Let $\mathcal{S}_B(\mathbb{S}_+^n)$ denote the set of spherical star bodies with respect to a spherical cap in $\mathbb{S}_+^n$ (see
Definition \ref{ssb}). Let ${\rm \Pi}^{\circ}_{\mathbb S}(K)$ denote the spherical polar body of spherical projection body of $K$ (see Definition \ref{spb} and Definition \ref{sprob}).

\begin{thm}\label{TSPB}
If $K\in\mathcal{S}_B(\mathbb{S}_+^n)$ and $K^\star$ is the spherical cap center at $e_{n+1}$ with the same measure as $K$, then
\begin{eqnarray}\label{ISPB}
\mathcal H^n\left({\rm \Pi}^{\circ}_{\mathbb S}(K)\right)\leq \mathcal H^n\left({\rm \Pi}^{\circ}_{\mathbb S}(K^\star)\right),
\end{eqnarray}
with the equality if and only if $K=K^\star$.
\end{thm}

In hyperbolic geometry, we will use Poincar\'e model. In this model, the hyperbolic space $\mathbb{H}^n$ is identified with the open Euclidean unit ball $\mathbb{B}^n$ equipped with a certain metric. In this paper, we give a $\Phi$ transformation from $\mathbb{B}^n$ to the Euclidean space $\mathbb{R}^n$. Combining Poincar\'e model and $\Phi$ transformation, we can define hyperbolic Steiner symmetrizations, hyperbolic projection bodies and hyperbolic polar bodies. Using the monotonic increasing property of the measure of the polar bodies of hyperbolic star bodies and the continuity of hyperbolic projection operator with respect to hyperbolic Hausdorff distance, we proved the hyperbolic Petty projection inequality.
Let $\mathcal{S}_B(\mathbb{H}^n)$ denote the set of  hyperbolic star bodies with respect to a hyperbolic ball in $\mathbb{H}^n$ (see Definition \ref{hsb}). Let ${\rm \Pi}^{\circ}_{\mathbb H}(K)$ denote the hyperbolic polar body of  hyperbolic projection body of $K$ (see Definition \ref{hpb} and Definition \ref{hprob}).

\begin{thm}\label{THPI}
If $K\in\mathcal{S}_B(\mathbb{H}^n)$ and $K^\star$ is the hyperbolic ball center at $e_{n+1}$ with the same measure as $K$, then
\begin{eqnarray}\label{IHPI}
\mathcal H^n\left({\rm \Pi}^{\circ}_{\mathbb H}(K)\right)\leq \mathcal H^n\left({\rm \Pi}_{\mathbb H}^{\circ}(K^\star)\right),
\end{eqnarray}
with the equality if and only if $K=K^\star$.
\end{thm}
\section{Preliminaries}
\subsection{Basic facts from Euclidean convex geometry}

Let $\mathbb{R}^n$ denote $n$-dimensional Euclidean space. Let $o$ denote the origin of $\mathbb{R}^n$. Let $e_1,\cdots, e_n$ denote the standard orthonormal basis of $\mathbb{R}^n$. Let $\mathbb{R}^{n-1}$ denote the subspace of $\mathbb{R}^n$ with the $n$-th component is $0$, i.e., $\mathbb{R}^{n-1}:=\{x\in\mathbb{R}^n,\;x\cdot e_n=0\}$. Let $\mathbb{S}^{n-1}$ denote the set of unit vectors of $\mathbb{R}^n$. Let $B_o(r)$ denote the closed ball centered at the origin $o$ with radius $r$ in $\mathbb{R}^n$.  Let $\mathcal{H}^k$ denote the $k$-dimensional Hausdorff measure. Let $\omega_n$ denote the volume of $B_o(1)$, i.e., $\omega_n:=\mathcal{H}^n(B_o(1))$. A convex body $K$ is a compact convex subset of $\mathbb{R}^n$. The set of convex bodies in $\mathbb{R}^n$ is denoted by $\mathcal{K}\left(\mathbb{R}^n\right)$. We denote by $\mathcal{K}_o\left(\mathbb{R}^n\right)$ the set of convex bodies that contain the origin in their interiors. Let $\|\cdot\|$ denote the Euclidean norm. For $K\in\mathcal{K}\left(\mathbb{R}^n\right)$, $K$ is uniquely determined by its {\it support function} $h_K$ defined by
$$
h_K(x):=\max \{x \cdot y: y \in K\}, \quad x \in \mathbb{R}^n.
$$
The support function is homogeneous of degree $1$, i.e.,
\begin{eqnarray}\label{support+1}
h_K(rx)=rh_K(x),\;\;{\rm for}\;r>0.
\end{eqnarray}
For $K\in\mathcal{K}_o\left(\mathbb{R}^n\right)$, its {\it radial function} is defined by
\begin{eqnarray}\label{Radialfunction}
\rho_K(x):=\max\{r>0:\;rx\in K\}, \quad x \in \mathbb{R}^n\backslash\{o\}.
\end{eqnarray}
The radial function is homogeneous of degree $-1$, i.e.,
\begin{eqnarray}\label{rho-1}
\rho_K(rx)=\frac{1}{r}\rho_K(x),\;\;{\rm for}\;r>0.
\end{eqnarray}

A compact set $K\subset\mathbb{R}^n$ is a {\it star-shaped set} with respect to the $z\in K$ if the intersection of every straight line through $z$ with $K$ is a line segment. Let $K\subset\mathbb{R}^n$ be a compact star shaped set with respect to $z\in K$, the radial function $\rho_{K,z}(\cdot):\mathbb{R}^n\backslash\{o\}\rightarrow\mathbb{R}$ is defined by
\begin{eqnarray}
\rho_{K,z}(x):=\max\{r\geq0:\;z+rx\in K\}.
\end{eqnarray}
 If $\rho_{K,z}$ is strictly positive and continuous,
then we call $K$ a {\it star body} with respect to the $z$, denotes the class of star bodies in $\mathbb{R}^n$ by $\mathcal{S}_z(\mathbb{R}^n)$. If $K\subset\mathbb{R}^n$ is a star body with respect to each point of ball $B_o(r)$, then we say $K$ is a {\it star body with respect to a ball}. The class of star bodies with respect to ball $B_o(r)$ will be denoted by $\mathcal{S}_B(\mathbb{R}^n)$. It is clear that $\mathcal{K}_o(\mathbb{R}^n)\subset\mathcal{S}_B(\mathbb{R}^n)$, i.e., any convex body with the origin as its interior is a star body with respect to a ball.
For $K\in\mathcal{S}_B(\mathbb{R}^n)$, we have the following volume formula
\begin{eqnarray}\label{volformula}
\mathcal{H}^n(K)=\frac{1}{n}\int_{\mathbb{S}^{n-1}}\rho^n_K(u)du
\end{eqnarray}
and the perimeter formula
\begin{eqnarray}\label{performula}
\mathcal{H}^{n-1}(\partial K)=\int_{\mathbb{S}^{n-1}}\rho^{n-1}_K(u)\left[u\cdot\nu^K\left(\rho_K(u)u\right)\right]^{-1}du,
\end{eqnarray}
where $\nu^K\left(\rho_K(u)u\right)$ denotes the outer unit normal vector of $K$ at the boundary point $\rho_K(u)u$.

For $K\in\mathcal{S}_o\left(\mathbb{R}^n\right)$, its {\it polar body} is defined by
$$K^{\ast}:=\{y\in\mathbb{R}^n:\;y\cdot x\leq 1\;\;{\rm for\;any}\;x\in K\}.$$
It is well-known that if $K\in\mathcal{K}_o\left(\mathbb{R}^n\right)$,
\begin{eqnarray}\label{Edoublestar}
\left(K^{\ast}\right)^{\ast}=K.
\end{eqnarray}
The support function and radial function of $K\in\mathcal{K}_o(\mathbb{R}^n)$ have the following relationship:
\begin{eqnarray}\label{hkrhokstar}
h_K(x)\rho_{K^{\ast}}(x)=1,\;\;\;\;x\in\mathbb{R}^n\backslash\{o\}.
\end{eqnarray}

The {\it radial distance} between $K, L\in\mathcal{S}_o(\mathbb{R}^n)$ is defined by
\begin{eqnarray}
d_R(K,L):=\left|\rho_K-\rho_L\right|_{\infty}=\max_{u\in\mathbb{S}^{n-1}}\left|\rho_K(u)-\rho_L(u)\right|.
\end{eqnarray}
The {\it Hausdorff distance} between the compact sets $K,L\subset \mathbb{R}^n$ is defined by
\begin{eqnarray}
d_H(K,L):=\min\left\{r\geq 0:\;K\subset L+B_o(r),\;L\subset K+B_o(r)\right\}.
\end{eqnarray}

For  $K\in\mathcal{S}_B\left(\mathbb{R}^n\right)$, its {\it Petty projection body} is defined with its support function:
\begin{eqnarray}\label{Dsupportfunction}
h_{{\rm \Pi}(K)}(z):=\frac{1}{2}\int_{\partial K}\left|\nu^K(x)\cdot z\right|d\mathcal{H}^{n-1}(x),
\end{eqnarray}
where $\partial K$ denotes the boundary of $K$, $\nu^K(x)$ denotes the unit outer normal vector of $K$ at the boundary point $x\in\partial K$, ``$\cdot$" denotes the Euclidean
scalar product and $\mathcal{H}^{n-1}$ denotes the $(n-1)$-Hausdorff measure. The polar body of ${\rm \Pi}(K)$ will be denoted by ${\rm \Pi}^{\ast}(K)$ rather than $({\rm \Pi}(K))^{\ast}$.

The following lemma shows that the Petty projection operator ${\rm \Pi}:\;\mathcal{S}_B\left(\mathbb{R}^n\right)\rightarrow\mathcal{K}_o(\mathbb{R}^n)$ is continuous
when $K_i$ converges to $K_0$ in the Hausdorff distance. 

\begin{lem}\label{projectionbodyconvergece}\cite[Proposition 4.1]{Lin21}
Let $K_0,K_i\in\mathcal{S}_B (\mathbb{R}^n)$, $i\in\mathbb{N}$. If $K_i\rightarrow K_0$ in the Hausdorff distance and $\mathcal{H}^{n-1}(\partial K_i)\rightarrow\mathcal{H}^{n-1}(\partial K_0)$, then ${\rm \Pi} K_i\rightarrow {\rm \Pi} K_0$ in the Hausdorff distance.
\end{lem}

For $K\in\mathcal{S}_o\left(\mathbb{R}^n\right)$, its {\it Steiner symmetrization} along the direction $u\in\mathbb{S}^{n-1}$ is defined by
\begin{eqnarray}\label{DESteiner}
S_uK:=\bigcup_{x'\in K|_{u^{\perp}}}\left\{x'+tu:\;t\in \left[-\frac{1}{2}\mathcal{H}^1(K_{u,x'}),\frac{1}{2}\mathcal{H}^1(K_{u,x'})\right]\right\},
\end{eqnarray}
where $u^{\perp}$ denotes the orthogonal complementary space of $u$, i.e., $u^{\perp}:=\{z\in\mathbb{R}^n:\;z\cdot u=0\}$; $K|_{u^{\perp}}$ denotes the orthogonal projection
of $K$ onto $u^{\perp}$, i.e.,
$$K|_{u^{\perp}}:=\{x'\in u^{\perp}:\;{\rm there\;exits\;some}\;r\in\mathbb{R}\;{\rm such\;that}\;x'+ru\in K\};$$ $K_{u,x'}$ denotes the intersection of $K$ and the straight line parallel to $u$ and passing through point $x'$, i.e., $$K_{u,x'}:=\{x\in K:\;x=x'+ru,\;r\in\mathbb{R}\}.$$
For simplicity of notation, we write $SK$, $K'$ and $K_{x'}$ instead of $S_{e_n}K$, $K|_{e_n^{\perp}}$ and $K_{e_n,x'}$, respectively.

Let $K\in\mathcal{S}_o\left(\mathbb{R}^n\right)$. Its {\it symmetric rearrangement} $K^{\star}$ is the closed centered ball whose volume agrees with $K$,
$$K^{\star}:=\left\{x\in\mathbb{R}^n:\;\omega_n\|x\|^n\leq\mathcal{H}^n(K)\right\}.$$
Let $f:\mathbb{R}^n\rightarrow[0,+\infty)$ be a nonnegative measurable function that {\it vanishes at infinity}, in the sense that all its positive level sets have finite
measure,
$$\mathcal{H}^n\left(\{x:\;f(x)>t\}\right)<\infty,\;\;{\rm for\;all}\;t>0.$$
We define the {\it symmetric decreasing rearrangement} $f^{\star}$ of $f$ by symmetrizing its the level sets,
\begin{eqnarray}\label{symmdecrerearr}
f^{\star}(x)=\int_0^{\infty}\chi_{\{f>t\}^{\star}}(x)dt,
\end{eqnarray}
where $\chi_E$ denote the characteristic function of $E\subset\mathbb{R}^n$, i.e.,
\begin{equation}
\chi_E(x)=\left\{\begin{aligned}
&1&&{\rm if}\;x\in E,\\
&0&&{\rm otherwise}.
\end{aligned} \right.
\end{equation}

The following lemma provides some properties on the Steiner symmetrizations of star bodies.

\begin{lem}\label{Starbodies}\cite[Lemma 5.1]{LW23}
If $K\in\mathcal{S}_B(\mathbb{R}^n)$, then $S_uK\in \mathcal{S}_B(\mathbb{R}^n)$ for every $u\in\mathbb{S}^{n-1}$.
\end{lem}

\begin{lem}\label{conHaurRad}\cite[Theorem 2.2]{LW23}
Let $K,K_i\in\mathcal{S}_B(\mathbb{R}^n)$, $i\in\mathbb{N}$. Then, the fact that $K_i$ converges to $K$ in Hausdorff distance is equivalent to the fact that $K_i$ converges to $K$ in radial distance.
\end{lem}

\begin{lem}\label{convergesstarbodies}\cite[Theroem 2.3]{LW23}
If $K\in\mathcal{S}_B(\mathbb{R}^n)$ and $T$ is a dense subset of $\mathbb{S}^{n-1}$, then there is a sequence $\{u_i\}\subset T$ such that
$K_i:=S_{u_i}\cdots S_{u_1}K$ converges to $K^{\star}$ in radial distance.
\end{lem}

The following lemma characterizes the structures of the boundaries of  star bodies with respect to a ball.
\begin{lem}\label{Starbodyboundary} \cite[Theorem 3.1]{LX22} Let $K\in\mathcal{S}_B(\mathbb{R}^n)$. Then, for almost all $u \in \mathbb{S}^{n-1}$, there is a sequence of disjoint open subsets $G_m \subset K^{\prime}$, and two sequences of graph functions
$f_{m,j}, g_{m,j}: G_m \rightarrow \mathbb{R},\;1\leq j\leq m,$ satisfying

\noindent (i) $\bigcup_{m=1}^{\infty} G_m$ is open dense in $K^{\prime}$, and $f_{m,1}<g_{m,1}<\cdots<f_{m,j}<g_{m,j}$;

\noindent (ii) $K$ has the representation (if we neglect an $\mathcal{H}^n$-null set)
$$
K=\bigcup_{m=1}^{\infty} \bigcup_{\substack{x^{\prime} \in G_m \\ 1 \leq j \leq m}}\left\{(x',x_n):f_{m,j}(x')\leq x_n\leq g_{m,j}(x')\right\}
$$
and $\partial K$ has the representation (if we neglect an $\mathcal{H}^{n-1}$-null set)
$$
\partial K=\bigcup_{m=1}^{\infty} \bigcup_{j=1}^m\left\{\left(x^{\prime}, f_{m,j}\left(x^{\prime}\right)\right): x^{\prime} \in G_m\right\} \cup\left\{\left(x^{\prime}, g_{m,j}\left(x^{\prime}\right)\right): x^{\prime} \in G_m\right\};
$$
(iii) $f_{m,j}$, $g_{m,j}$ are differentiable at each $x'\in G_m$, and
$$\nu^K(x',f_{m,j}(x'))=\frac{(\nabla f_{m,j}(x'),-1)}{\sqrt{1+\left|\nabla f_{m,j}(x')\right|^2}},\;\;\;\nu^K(x',g_{m,j}(x'))=\frac{(-\nabla g_{m,j}(x'),1)}{\sqrt{1+\left|\nabla g_{m,j}(x')\right|^2}}.$$
\end{lem}

The following well-known fact, which is provided by Lutwak, Yang and Zhang \cite{MR2563216}, establishes the  relationship between Steiner symmetrizations and polar bodies.

\begin{lem}\label{Lprojectionbody}\cite[Lemma 1.1.]{MR2563216}
For two convex bodies $K,L\in\mathcal{K}_o(\mathbb{R}^n)$,
$$S_{e_n}L^{\ast}\subset K^{\ast}$$
 if and only if
$$h_L(z',t)=h_L(z',-s)=1,\;{\rm with}\;t\neq -s\;\;\Longrightarrow\;\; h_K\left(z',\frac{s+t}{2}\right)\leq 1.$$ In addition, if $S_{e_n}L^{\ast}=K^{\ast}$, then
$h_K(z',\frac{s+t}{2})=1$ for any $(z',t),(z',-s)\in\mathbb{R}^{n-1}\times\mathbb{R}$ satisfying  $t\neq -s$ and $h_L(z',t)=h_L(z',-s)=1$.
\end{lem}

\subsection{Basic facts from Spherical convex geometry}
Let $\mathbb{R}^{n+1}$ denote $(n+1)$-dimensional Euclidean space. Let $\mathbb{R}^n:=\left\{x \in \mathbb{R}^{n+1}: x\cdot e_{n+1}=0\right\}$. Let $o$ denote the origin of $\mathbb{R}^{n+1}$. Let $e_1,\cdots, e_n,e_{n+1}$ denote the standard orthonormal basis of $\mathbb{R}^{n+1}$.
We denote the Euclidean unit sphere in $\mathbb{R}^{n+1}$ by $\mathbb{S}^n, n \geq 2$. For $u\in\mathbb{S}^n$, let $\mathbb{S}_u$ denote the set $\{v\in\mathbb{S}^n:\;v\cdot u=0\}$. Let $\mathbb{S}_u^{+}$ denote the set $\{v\in\mathbb{S}^n:\;v\cdot u>0\}$ and $\mathbb{S}_u^{-}$ denote the set $\{v\in\mathbb{S}^n:\;v\cdot u<0\}$. To simplify notation, we let $\mathbb{S}^{n-1}$ denote $\mathbb{S}_{e_{n+1}}$ and $\mathbb{S}^n_{\pm}$ denote $\mathbb{S}^{\pm}_{e_{n+1}}$. A set $A \subseteq \mathbb{S}^n$ is called {\it (spherical) convex} if its radial extension
$$
\operatorname{rad} A=\left\{r v \in \mathbb{R}^{n+1}: r \geq 0 \text { and } v \in A\right\}
$$
is convex in $\mathbb{R}^{n+1}$.
A closed convex subset of $\mathbb{S}^n$ is called a {\it(spherical) convex body}. The set of convex bodies is denoted by $\mathcal{K}\left(\mathbb{S}^n\right)$. Furthermore, the set of convex bodies contained in $\mathbb{S}^n_+$ with $e_{n+1}$ as its interior point is denoted by $\mathcal{K}_o\left(\mathbb{S}_+^n\right)$. And the set of convex bodies contained in $\mathbb{S}^n_-$ with $-e_{n+1}$ as its interior point is denoted by $\mathcal{K}_o\left(\mathbb{S}_-^n\right)$.
%

The natural spherica distance $d_s$ is given by $d_s(u, v)=\arccos (u \cdot v)$ for $u, v \in \mathbb{S}^n$. For spherical compact sets $K,L\subset\mathbb{S}^n$,  the {\it spherical Hausdorff distance} of $K$ and $L$ is defined by
\begin{eqnarray}
d_s(K,L):=\inf\left\{r>0:\;K\subseteq L_{r}\;\;{\rm and}\;\;L\subseteq K_r\right\},
\end{eqnarray}
where $L_r$ denotes the spherical parallel set of $L$, which is defined by
$$L_r:=\{w\in\mathbb{S}^n:\;{\rm there\;exists}\;v\in L\;{\rm such\;that}\;d_s(w,v)\leq r\}.$$
 Let $B_s(\alpha)$ denote the spherical cap of radius $\alpha\in(0,\pi/2)$ and center $e_{n+1}$ in $\mathbb{S}^n$, i.e.,
$$B_s(\alpha):=\{v\in\mathbb{S}^n:\;d_s(v,e_{n+1})\leq \alpha\}.$$

The {\it convex hull} $\operatorname{conv} A$ of $A \subseteq \mathbb{S}^n$ is the intersection of all convex bodies in $\mathbb{S}^n$ that contain $A$. The convex hull of two spherical convex bodies $K, L$ is denoted by $\operatorname{conv}(K, L)$, i.e., $$\operatorname{conv}(K, L):=\operatorname{conv}(K \cup L).$$ The segment spanned by two points $u, v \in \mathbb{S}^n, u \neq-v$, is given by $\operatorname{conv}(u, v)=\operatorname{conv}(\{u\},\{v\})$.  A $k$-sphere $S$, $k \in\{0, \ldots, n\}$, is the intersection of a $(k+1)$-dimensional linear subspace of $\mathbb{R}^{n+1}$ with $\mathbb{S}^n$.   Let $S$ be a $k$-sphere and let $K \in \mathcal{K}\left(\mathbb{S}^n\right)$. Then the {\it spherical projection} $K \mid S$ is defined by
$$
K \mid S:=\operatorname{conv}\left(S^{\circ},K\right) \cap S,
$$
where $S^{\circ}:=\{w\in\mathbb{S}^n:\;w\cdot u=0\;{\rm for\;all}\;u\in S\}$. The spherical projection about point is given by $x\mid S:=\{x\}\mid S$.

\begin{definition} For $K \in \mathcal{K}_o\left(\mathbb{S}_+^n\right)$, the {\it spherical support function} $h(K, \cdot): \mathbb{S}^{n-1} \rightarrow$ $\left(0, \frac{{\pi}}{2}\right)$ of $K$ is defined by
\begin{eqnarray}
h_s(K, v)=\max \left\{ \operatorname{sgn}(v\cdot w)d_s\left(e_{n+1}, w \mid \mathbb{S}_{e_{n+1}, v}^1\right): w \in K\right\},\;\;\;v\in\mathbb{S}^{n-1},
\end{eqnarray}
where $\mathbb{S}_{e_{n+1}, v}^1$ denotes the 1-sphere spanned by $e_{n+1}$ and $v$.
\end{definition}

The intuitive interpretation of the spherical support function is as follows: If $v \in \mathbb{S}^{n-1}$ then the projection $K \mid \mathbb{S}_{e_{n+1}, v}^1$ is a spherical segment and the spherical support function measures the width along the direction $v$ with respect to $e_{n+1}$. We have
$$
K \mid \mathbb{S}_{e_{n+1}, v}^1=\left\{\cos (\alpha) e_{n+1}+\sin (\alpha) v: \alpha \in\left[-h(K,-v), h(K, v)\right]\right\}.
$$

\begin{definition}\label{DSRF} For $K\in\mathcal{K}_o(\mathbb{S}_+^n)$, its {\it spherical radial function} is defined by
\begin{eqnarray}\label{dsrf}
\rho_s(K,v):=\max\left\{\operatorname{sgn}(v\cdot w)d_s\left(e_{n+1},w\right):w\in K\cap\mathbb{S}^1_{e_{n+1},v}\;\right\},\;\;\;v\in\mathbb{S}^{n-1}.
\end{eqnarray}
\end{definition}

\begin{definition}\label{spb} For $K\in\mathcal{K}_o(\mathbb{S}_+^n)$, its {\it spherical polar body} $K^{\circ}$ is defined by
\begin{eqnarray}\label{EDSPolarbody}
K^{\circ}=\{v\in\mathbb{S}^n:\;v\cdot w\leq 0\;\;{\rm for\;all}\;w\in K\}.
\end{eqnarray}
\end{definition}
By the above definition, if $K\in\mathcal{K}_o(\mathbb{S}_+^n)$, then $K^{\circ}\in\mathcal{K}_o(\mathbb{S}_-^n)$.

\begin{definition}\label{Dgnom} The {\it gnomonic projection} $g: \mathbb{S}^n_+ \rightarrow \mathbb{R}^n$ is defined by
$$
g(v):=\frac{v}{e_{n+1} \cdot v}-e_{n+1}.
$$
The {\it inverse gnomonic projection} $g^{-1}: \mathbb{R}^n \rightarrow \mathbb{S}^n_+$ is defined by
$$g^{-1}(x):=\frac{x+e_{n+1}}{\|x+e_{n+1}\|}.$$
\end{definition}

By the definition of gnomonic projection of $K\in\mathcal{K}_o(\mathbb{S}_+^n)$, the following equalities show the relations between spherical support function (spherical radial function) of $K$ and the Euclidean support function (Euclidean radial function) of $g(K)$:
\begin{eqnarray}\label{EhgK}
h(g(K),v)=\tan h_s(K,v),\;\;\;v\in \mathbb{S}^{n-1},
\end{eqnarray}
\begin{eqnarray}\label{ErhogK}
\rho(g(K),v)=\tan \rho_s(K,v),\;\;\;v\in \mathbb{S}^{n-1}.
\end{eqnarray}
Moreover, it is easily to prove that for $K\in\mathcal{K}_o(\mathbb{S}_+^n)$,
\begin{eqnarray}\label{EgKast}
g(K)^{\ast}=-g(-K^{\circ}).
\end{eqnarray}
By (\ref{Edoublestar}) and (\ref{EgKast}), for $K\in\mathcal{K}_o(\mathbb{S}_+^n)$, we have
\begin{eqnarray}
\left(K^{\circ}\right)^{\circ}=K.
\end{eqnarray}
By (\ref{EhgK}), (\ref{ErhogK}) and (\ref{EgKast}),
$$\tan h_s(K,v)=h(g(K),v)=\frac{1}{\rho(g(K)^{\ast},v)}=\frac{1}{\rho(g(-K^{\circ}),-v)}=\frac{1}{\tan \rho_s(-K^{\circ},-v)}.$$
Therefore, for $K\in\mathcal{K}_o(\mathbb{S}_+^n)$ and $v\in\mathbb{S}^{n-1}$,
\begin{eqnarray}\label{Rhandrho}
h_s(K,v)+\rho_s(-K^{\circ},-v)=\frac{{\pi}}{2}.
\end{eqnarray}

By Definition \ref{Dgnom} and the definitions of star bodies in $\mathbb{R}^n$, we define spherical star bodies as follows.

\begin{definition}\label{ssb}
For a spherical compact set $K\subset\mathbb{S}_+^n$, if its gnomonic projection $g(K)$ is a star body with respect to $o$ in $\mathbb{R}^n$,
 then $K$ is called as {\it spherical star body} with respect to $e_{n+1}$. If $g(K)$ is a star body with respect to a ball $B_o(\tan \alpha)$ in $\mathbb{R}^n$, then $K$ is called as {\it spherical star body} with respect to a spherical cap $B_s(\alpha)$.
\end{definition}

The set of spherical star bodies with respect to $e_{n+1}$ is denoted
  by $\mathcal{S}_o(\mathbb{S}_+^n)$. The set of spherical star bodies with respect to $B_s(\alpha)$ is denoted
  by $\mathcal{S}_B(\mathbb{S}_+^n)$. Similarly, $\mathcal{K}_o(\mathbb{S}_+^n)\subset\mathcal{S}_B(\mathbb{S}_+^n)$. For $K\in \mathcal{S}_o(\mathbb{S}_+^n)$, its spherical radial function $\rho_K$ can be defined as in (\ref{dsrf}).

\begin{definition} For $\bar{K}\in\mathcal{K}_o(\mathbb{R}^n)$, its {\it spherical measure} is defined by
\begin{eqnarray}\label{Ms}
M_s(\bar{K}):=\int_{\bar{K}}\left(1+\|x\|^2\right)^{-\frac{n+1}{2}}dx.
\end{eqnarray}
\end{definition}

By \cite[Lemma 2.3]{BHPS23}, the spherical measure of $\bar{K}\in\mathcal{K}_o(\mathbb{R}^n)$ equals the $n$-Hausdorff measure of its inverse gnomonic projection $\mathcal{H}^n\left(g^{-1}(\bar{K})\right)$. Thus, for $K\in\mathcal{K}_o(\mathbb{S}_+^n)$, we have
\begin{eqnarray}\label{EKandgK}
\mathcal{H}^n(K)=M_s(g(K)).
\end{eqnarray}

\begin{lem}\label{Lscoarea} (\cite[Lemma 6.5.1]{SW08}). Let $S$ be a $k$-sphere, $0 \leq k \leq n-1$, and let $f: \mathbb{S}^n \rightarrow \mathbb{R}$ be a non-negative measurable function. Then
\begin{eqnarray}
\int_{\mathbb{S}^n} f(w) d w=\int_S \int_{\operatorname{conv}\left(S^{\circ}, v\right)} \sin \left(d_s\left(S^{\circ}, u\right)\right)^k f(u) d u d v.
\end{eqnarray}
\end{lem}

\subsection{Basic facts from Hyperbolic convex geometry}

Let us recall some facts about hyperbolic geometry; see e.g. \cite{Ratcliffe94}.
In $\mathbb{R}^{n+1}$, let $$\mathbb{H}^n:=\left\{(x,x_{n+1})\in\mathbb{R}^{n+1}:\;\|x\|^2-x^2_{n+1}=-1,\;x_{n+1}>0\right\}$$
denote the upper sheet of a two-sheet hyperboloid.

\begin{definition}\label{DPoincare}(Poincar\'e model) Let $o^{\prime}:=-e_{n+1}$. For any $\bar{x}:=(x,x_{n+1})\in\mathbb{H}^n$, {\it Poincar\'e model projection point} of $\bar{x}$, denoted by $P(\bar{x})$, is the intersection of the half-line $o'\bar{x}$ and $\mathbb{R}^n$. In {\it Poincar\'e model}, $\mathbb{H}^n$ is identified with the following open unit ball equipped with a certain metric
\begin{eqnarray}\label{DBKMB}
\mathbb{B}^n:=\{(x,x_{n+1})\in\mathbb{R}^{n+1}:\;\|x\|<1,\;x_{n+1}=0\}.
\end{eqnarray} 
\end{definition}

We call the projection $P:\;\mathbb{H}^n\rightarrow \mathbb{B}^n$ as {\it Poincar\'e model projection}. In Poincar\'e model, the corresponding metric is
\begin{eqnarray}\label{Metric}
ds^2=4\frac{dx_1^2+\dots+dx_n^2}{\left(1-\left(x_1^2+\dots+x_n^2\right)\right)^2}.
\end{eqnarray}
In this metric, geodesic segments are arcs of the circles orthogonal to the boundary of the ball $\mathbb{B}^n$. If a segment passes through the origin, then
the circle becomes a straight line. We say that a body $\bar{K}\subset\mathbb{B}^n$ is {\it $h$-convex} if it is convex with respect to the metric (\ref{Metric}). The
$h$-convexity means that for any two points $x$ and $y$ in $\bar{K}$ the geodesic segment connecting these two points is also in $\bar{K}$.

For two compact sets $K,L\subset \mathbb{H}^n$, their {\it hyperbolic Hausdorff distance} is defined by
\begin{eqnarray}
d_h(K,L):=\inf\left\{r>0:\;K\subset L_{r},\;\;L\subset K_r\right\},
\end{eqnarray}
where
$$K_r:=\{\bar{y}\in\mathbb{H}^n:\;{\rm there\;exits}\;\bar{x}\in K\;{\rm such\;that}\;ds^2(P(\bar{x}),P(\bar{y}))\leq r^2\}.$$
Let $B_h(\alpha)\subset\mathbb{H}^n$ denote the {\it hyperbolic ball} centered at $e_{n+1}$ with radius $\alpha$, i.e.,
$$B_h(\alpha):=\left\{\bar{x}\in\mathbb{H}^n:\;ds^2\left(P(\bar{x}),o\right)\leq \alpha^2\right\}.$$

The volume element of the metric (\ref{Metric}) equals
\begin{eqnarray}\label{dmun}
d\mu_n=2^n\frac{dx_1\cdots dx_n}{\left(1-(x_1^2+\cdots+x_n^2)\right)^n}=2^n\frac{dx}{\left(1-\|x\|^2\right)^n}.
\end{eqnarray}
Therefore, for $\bar{K}\subset\mathbb{B}^n$, the hyperbolic volume  is then given by
\begin{eqnarray}
\operatorname{hvol}_n(\bar{K})=\int_{\bar{K}}d\mu_n=2^n\int_{\bar{K}}\frac{dx}{\left(1-\|x\|^2\right)^n}.
\end{eqnarray}
If $\bar{K}$ is a star body in $\mathbb{B}^n$, we can write its hyperbolic volume in polar coordinates,
\begin{eqnarray}\label{hvol}
\operatorname{hvol}_n(\bar{K})=2^n\int_{\mathbb{S}^{n-1}}\int_{0}^{\rho_{\bar{K}}(u)}\frac{r^{n-1}}{\left(1-r^2\right)^n}drdu,
\end{eqnarray}
where $\rho_{\bar{K}}$ denotes the radial function of $\bar{K}$ in $\mathbb{R}^n$ given in (\ref{Radialfunction}).

\begin{definition}\label{hsb} For $K\subset\mathbb{H}^n$, if $P(K)$ is a star body with respect to the origin in $\mathbb{R}^n$, then we call $K$ is a {\it hyperbolic star body} with respect to $e_{n+1}$. If $P(K)$ is a star body with respect to a ball $B_o(r)$ in $\mathbb{R}^n$, then $K$ is called as a {\it hyperbolic star body with respect to some hyperbolic ball}. \end{definition}

 We denote the class of hyperbolic star bodies  with respect to $e_{n+1}$ in $\mathbb{H}^n$ by $\mathcal{S}_o(\mathbb{H}^n)$ and denote the class of hyperbolic star bodies  with respect to a hyperbolic ball by $\mathcal{S}_B(\mathbb{H}^n)$.

\section{Spherical projection body and Spherical projection inequality}

\subsection{Spherical Steiner symmetrization}
In this section, we define the spherical Steiner symmetrization for spherical star bodies and study its some fundamental  properties.
Without loss of generality, we only consider the Steiner symmetrization along the direction $e_n$. If $K\in\mathcal{S}_B(\mathbb{S}_+^n)$, then $g(K)\in\mathcal{S}_B(\mathbb{R}^n)$. By Lemma \ref{Starbodyboundary}, there is a sequence of disjoint open subsets $G_m \subset K^{\prime}$, and two sequences of graph functions
$$
f_{m,j}, g_{m,j}: G_m \rightarrow \mathbb{R},\;1\leq j\leq m,
$$ satisfying (i), (ii) and (iii) in Lemma \ref{Starbodyboundary}. In particular,
$g(K)$ has the representation (if we neglect an $\mathcal{H}^n$-null set)
\begin{eqnarray}\label{gKfmj}
g(K)=\bigcup_{m=1}^{\infty} \bigcup_{\substack{x^{\prime} \in G_m \\ 1 \leq j \leq m}}\left\{(x',x_n):f_{m,j}(x')\leq x_n\leq g_{m,j}(x')\right\}.
\end{eqnarray}

By (\ref{DESteiner}), the Euclidean Steiner symmetrization $Sg(K)$ of $g(K)$ along the direction $e_n$ (if we neglect an $\mathcal{H}^n$-null set)
\begin{eqnarray}\label{overunderSgK}
Sg(K)=\left\{(x',x_n)\in\mathbb{R}^n:\;\;x'\in \bigcup_{m=1}^{\infty}G_m,\;\;\underline{\varrho}(x')\leq x_n\leq\overline{\varrho}(x')\right\},
\end{eqnarray}
where for $x'\in G_m$,
\begin{eqnarray}\label{gKandSgK}
\overline{\varrho}(x')=\sum_{j=1}^{m}\frac{g_{m,j}(x')-f_{m,j}(x')}{2}=-\underline{\varrho}(x').
\end{eqnarray}

The following lemma shows that the spherical measure of $Sg(K)$ is not less than the spherical measure of $g(K)$.

\begin{lem}\label{LMeasure}
For convex body $K\in\mathcal{K}_o(\mathbb{S}^n)$, we have
\begin{eqnarray}\label{ESgK}
M_s\left(Sg(K)\right)\geq M_s\left(g(K)\right),
\end{eqnarray}
with the equality if and only if $Sg(K)=g(K)$.
\end{lem}
\begin{proof}
By Fubini's theorem and the definition of spherical measure (\ref{Ms}), we only need to prove that for any $x'\in g(K)|_{\mathbb{R}^{n-1}}$,
\begin{eqnarray}\label{Onlyprove}
\int_{(Sg(K))_{x'}}\left(1+\|x'\|^2+|x_n|^2\right)^{-\frac{n+1}{2}}dx_n\geq \int_{g(K)_{x'}}\left(1+\|x'\|^2+|x_n|^2\right)^{-\frac{n+1}{2}}dx_n.
\end{eqnarray}
Let
$$f_1(x_n):=\left(1+\|x'\|^2+|x_n|^2\right)^{-\frac{n+1}{2}},\;\;\;f_2(x_n):=\chi_{g(K)_{x'}}(x_n).$$
Then
$$f_1^{\star}=f_1,\;\;{\rm and}\;\;f_2^{\star}=\chi_{Sg(K)_{x'}},$$
where $f^{\star}$ denotes the symmetric decreasing rearrangement of $f$ (see (\ref{symmdecrerearr}) for specific definition).

By Hardy-Littlewood inequality, we have
\begin{eqnarray}\label{Littlewood}
&&\int_{g(K)_{x'}}\left(1+\|x'\|^2+|x_n|^2\right)^{-\frac{n+1}{2}}dx_n\\
&=&\int_{\mathbb{R}}f_1(x_n)f_2(x_n)dx_n\nonumber\\
&\leq&\int_{\mathbb{R}}f^{\star}_1(x_n)f^{\star}_2(x_n)dx_n\nonumber\\
&=&\int_{(Sg(K))_{x'}}\left(1+\|x'\|^2+|x_n|^2\right)^{-\frac{n+1}{2}}dx_n.\nonumber
\end{eqnarray}
Moreover, since $f_1$ is an even nonnegative unimodal integrable function, the equality in (\ref{Littlewood}) holds if and only if $g(K)_{x'}=Sg(K)_{x'}$. Thus, the equality in (\ref{ESgK}) holds if and only if $Sg(K)=g(K)$.
\qed
\end{proof}

By Lemma \ref{LMeasure}, $M_s\left(Sg(K)\right)\geq M_s\left(g(K)\right)$. Thus, there exists some real number $r_K\in(0,1]$ such that
\begin{eqnarray}\label{DbarS}
\bar{S}g(K):=\left\{(x',x_n)\in\mathbb{R}^n:\;\;x'\in K|_{\mathbb{R}^{n-1}},\;\;r_K\underline{\varrho}(x')\leq x_n\leq r_K\overline{\varrho}(x')\right\}
\end{eqnarray}
satisfies
\begin{eqnarray}\label{EbarS}
M_s\left(\bar{S}g(K)\right)=M_s\left(g(K)\right).
\end{eqnarray}

\begin{definition}\label{DStei}
For $K\in\mathcal{K}_o(\mathbb{S}^n)$, its {\it spherical Steiner symmetrization} $\hat{S}_{e_n}K$ along the direction $e_n\in \mathbb{S}^{n-1}$ is defined by
\begin{eqnarray}\label{DSteiner}
\hat{S}_{e_n}(K):=g^{-1}\left(\bar{S}g(K)\right),
\end{eqnarray}
where $\bar{S}g(K)$ is defined in (\ref{DbarS}).
\end{definition}

For simplicity of notation, we write $\hat{S}K$ instead of $\hat{S}_{e_n}(K)$. In Definition \ref{DStei}, $g$ and $g^{-1}$ denote the gnomonic projection and the inverse gnomonic projection (see Definition \ref{Dgnom}). If $K\in\mathcal{K}_o(\mathbb{S}^n)$, then $g(K)\in\mathcal{K}_o(\mathbb{R}^n)$. Thus $Sg(K)\in\mathcal{K}_o(\mathbb{R}^n)$. By (\ref{DbarS}), $\bar{S}g(K)\in\mathcal{K}_o(\mathbb{R}^n)$. Thus by (\ref{DSteiner}),
\begin{eqnarray}
K\in\mathcal{K}_o(\mathbb{S}_+^n)\Longrightarrow\hat{S}K\in\mathcal{K}_o(\mathbb{S}_+^n).
\end{eqnarray}
 By Lemma \ref{Starbodies}, (\ref{DbarS}), (\ref{DSteiner}) and the definition of spherical star bodies with respect to a spherical cap (see Definition \ref{ssb}),
\begin{eqnarray}
K\in\mathcal{S}_B(\mathbb{S}_+^n)\Longrightarrow\hat{S}K\in\mathcal{S}_B(\mathbb{S}_+^n).
\end{eqnarray}
By (\ref{EKandgK}) and (\ref{EbarS}), the spherical Steiner symmetrization maintains the invariance of $n$-Hausdorff measure, i.e.,
\begin{eqnarray}\label{maintainmeasure}
\mathcal{H}^n(\hat{S}K)=\mathcal{H}^n(K).
\end{eqnarray}

Similarly, for compact set $K\subset\mathbb{S}^n$, we define the {\it spherical symmetric rearrangement} $K^{\star}$ as following
\begin{eqnarray}
K^{\star}:=\left\{v\in\mathbb{S}^n:\;d_s(v,e_{n+1})\leq \alpha,\;\;\mathcal{H}^n(K)=\mathcal{H}^n(B_s(\alpha))\right\}.
\end{eqnarray}

\begin{lem}\label{Pconverge}
For $K\in\mathcal{S}_B(\mathbb{S}_+^n)$, there exists a sequence of directions $\{u_i\}_{i=1}^{\infty}\subset\mathbb{S}^{n-1}$ such that the sequence of successive spherical
Steiner symmetrizations of $K$ converges to  $K^{\star}$ in  spherical Hausdorff distance, i.e.,
\begin{eqnarray}
\lim_{i\rightarrow\infty}d_s(\hat{S}_{u_i}\cdots\hat{S}_{u_1}(K),K^{\star})=0.
\end{eqnarray}
\end{lem}

\begin{proof}
By Lemma \ref{convergesstarbodies}, there exists a sequence of directions $\{u_i\}_{i=1}^{\infty}\subset\mathbb{S}^{n-1}$ such that
\begin{eqnarray}\label{Elim}
\lim_{i\rightarrow\infty}d_R(S_{u_i}\cdots S_{u_1}(g(K)),g(K)^{\star})=0.
\end{eqnarray}
Let $r_1\in (0,1]$ satisfy
$$M_s\left(\bar{S}_{u_1,r_1}(g(K))\right)=\mathcal{H}^n(K),$$
where $\bar{S}_{u_1,r_1}(g(K))$ denotes the star body with the overgraph function on the direction $u_1$
$$\overline{\varrho}_{u_1}\left(\bar{S}_{u_1,r_1}(g(K)),\cdot\right)=r_1\overline{\varrho}_{u_1}\left(S_{u_1}(g(K)),\cdot\right)$$
and the undergraph function on the direction $u_1$
$$\underline{\varrho}_{u_1}\left(\bar{S}_{u_1,r_1}(g(K)),\cdot\right)=r_1\underline{\varrho}_{u_1}\left(S_{u_1}(g(K)),\cdot\right).$$
Let $r_1'\in (0,1]$ satisfy
$M_s\left(\bar{S}_{u_1,r_1}(g(K))\right)=M_s\left(r_1'g(K)^{\star}\right)$.
Let $r_2\in (0,1]$ satisfy
$$M_s\left(\bar{S}_{u_2,r_2}\left(g(\hat{S}_{u_1}(K))\right)\right)=M_s(\hat{S}_{u_1}(K))$$.
Let $r_2'\in (0,1]$ satisfy
$M_s\left(\bar{S}_{u_2,r_2}(g(\hat{S}_{u_1}(K)))\right)=M_s\left(r_2'r_1'g(K)^{\star}\right)$.
Repeating the previous process,  we can get a sequence of real numbers $\{r_i\}_{i=1}^{\infty}\subset(0,1]$ such that
\begin{eqnarray}
M_s\left(\bar{S}_{u_i,r_i}\left(g\left(\hat{S}_{u_{i-1}}\cdots\hat{S}_{u_1}(K)\right)\right)\right)
=M_s\left(\hat{S}_{u_{i-1}}\cdots\hat{S}_{u_1}(K)\right),\nonumber
\end{eqnarray}
and a sequence of real numbers  $\{r'_i\}_{i=1}^{\infty}\subset(0,1]$ such that
\begin{eqnarray}\label{Msleft}
M_s\left(\bar{S}_{u_i,r_i}\left(g\left(\hat{S}_{u_{i-1}}\cdots\hat{S}_{u_1}(K)\right)\right)\right)=M_s\left(r_i'\cdots r_2'r_1'g(K)^{\star}\right).
\end{eqnarray}

Since $r'_i\in(0,1]$, the product of $r'_i\cdots r'_2r'_1$ is monotonically decreasing on $i$ and  $r'_i\cdots r'_2r'_1\in(0,1]$. Let
$\bar{r}:=\lim_{i\rightarrow\infty}\left(r'_ir'_{i-1}\cdots r'_1\right)$.
Since spherical Steiner symmetrization maintains the invariance of $n$-Hausdorff measure, $\bar{r}\in (0,1]$.
Thus, by the definition of spherical Steiner symmetrizations, the continuity of inverse gnomonic projection and (\ref{Elim}), we have
\begin{eqnarray}
&&\lim_{i\rightarrow\infty}d_s\left(\hat{S}_{u_i}\cdots \hat{S}_{u_1}(K),g^{-1}\left(r'_i\cdots r'_1g(K)^{\star}\right)\right)\nonumber\\
&=&\lim_{i\rightarrow\infty}d_s\left(g^{-1}\left(\bar{S}_{u_i,r_i}\left(g\left(\hat{S}_{u_{i-1}}\cdots\hat{S}_{u_1}(K)\right)\right)\right),g^{-1}\left(r'_i\cdots,r'_1g(K)^{\star}\right)\right)\nonumber\\
&\vdots&\;\;\;\;\;\;\;\;\;\;\;\;\;\;\;\;\;\;\;\;\;\;\;\;\;\;\;\;\;\;\;\;\nonumber\\
&=&\lim_{i\rightarrow\infty}d_s\left(g^{-1}\left(\bar{S}_{u_i,r_i}\cdots \bar{S}_{u_1,r_1}(g(K))\right),g^{-1}\left(r'_i\cdots r'_1g(K)^{\star}\right)\right)\nonumber\\
&\leq&\lim_{i\rightarrow\infty}d_R\left(\bar{S}_{u_i,r_i}\cdots \bar{S}_{u_1,r_1}(g(K)),r'_i\cdots r'_1g(K)^{\star}\right)=0.\nonumber
\end{eqnarray}
By (\ref{maintainmeasure}) and (\ref{Msleft}), $K^{\star}=g^{-1}\left(\bar{r}g(K)^{\star}\right)$. This is the desired conclusion.
\end{proof}
\subsection{Spherical projection body}

In this section, using the Euclidean projection bodies, we introduce the notion of spherical projection bodies and study some elementary properties.

\begin{definition}\label{sprob} For $K\in\mathcal{S}_B(\mathbb{S}_+^n)$, its {\it spherical projection body} ${\rm \Pi}_{\mathbb{S}}(K)$ is defined by
\begin{eqnarray}\label{EDSPB}
{\rm \Pi}_{\mathbb{S}}K:=g^{-1}\left({\rm \Pi}g(K)\right).
\end{eqnarray}
\end{definition}

By the definition of spherical projection body (\ref{EDSPB}), (\ref{EhgK}) and (\ref{Dsupportfunction}), for $u\in\mathbb{S}^{n-1}$,
\begin{eqnarray}\label{Etanh}
\tan h_s\left({\rm \Pi}_{\mathbb{S}}K,u\right)=\frac{1}{2}\int_{\partial g(K)}\left|u\cdot \nu^{g(K)}(y)\right|d\mathcal{H}^{n-1}(y).
\end{eqnarray}
The following lemma shows that the spherical projection operator ${\rm \Pi}_{\mathbb{S}}:\;\mathcal{S}_B(\mathbb{S}_+^n)\rightarrow\mathcal{K}_o(\mathbb{S}^n)$ is continuous.
\begin{lem}\label{Lcontinuous}
For a sequence of spherical star bodies $\{K_i\}_{i=0}^{\infty}\subset\mathcal{S}_B(\mathbb{S}_+^n)$, if
\begin{eqnarray}\label{limki}
\lim_{i\rightarrow\infty}d_s(K_i,K_0)=0,
\end{eqnarray}
then
\begin{eqnarray}\label{limdisPi}
\lim_{i\rightarrow\infty}d_s({\rm \Pi}_{\mathbb{S}}K_i,{\rm \Pi}_{\mathbb{S}}K_0)=0.
\end{eqnarray}
\end{lem}
\begin{proof}
By the continuity of gnomonic projection and (\ref{limki}), we have
\begin{eqnarray}\label{limdHgKi}
\lim_{i\rightarrow\infty}d_H(g(K_i),g(K_0))=0.
\end{eqnarray}
Since $\{K_i\}_{i=0}^{\infty}\subset\mathcal{S}_B(\mathbb{S}_+^n)$, $g(K_i)\in\mathcal{S}_B(\mathbb{R}^n)$. By Lemma \ref{conHaurRad} and (\ref{limdHgKi}),
$g(K_i)$ converges to $g(K_0)$ in radial distance. By (\ref{performula}) and Lemma \ref{projectionbodyconvergece}, we have
$$\lim_{i\rightarrow\infty}d_H\left({\rm \Pi}(g(K_i)),{\rm \Pi}(g(K_0))\right)=0.$$
Thus by the definition of spherical projection body and the continuity of gnomonic projection, we have
\begin{eqnarray*}
\lim_{i\rightarrow\infty}d_s\left({\rm \Pi}_{\mathbb{S}}K_i,{\rm \Pi}_{\mathbb{S}}K_0\right)&=&\lim_{i\rightarrow\infty}d_s\left(g^{-1}\left({\rm \Pi}(g(K_i))\right),g^{-1}\left({\rm \Pi}(g(K_0))\right)\right)\nonumber\\
&\leq &\lim_{i\rightarrow\infty}d_H\left({\rm \Pi}(g(K_i)),{\rm \Pi}(g(K_0))\right)\nonumber\\
&=&0.
\end{eqnarray*}
This is the desired conclusion.
\end{proof}

Let $\bar{\operatorname{O}}(n+1)$ denote the set of rotation transformations around the $x_{n+1}$-axis in $\mathbb{R}^{n+1}$. The following lemma demonstrates the rotation covariance of the spherical projection operator.

\begin{lem}
Let $\phi\in \bar{\operatorname{O}}(n+1)$ be a rotation transformation on $\mathbb{R}^{n+1}$ and $K\in\mathcal{S}_B(\mathbb{S}^n)$. Then
\begin{eqnarray}
{\rm \Pi}_{\mathbb{S}}(\phi K)=\phi{\rm \Pi}_{\mathbb{S}}K.
\end{eqnarray}
\end{lem}
\begin{proof}
For $\phi\in \bar{\operatorname{O}}(n+1)$, there exists a rotation transformation $\bar{\phi}\in\operatorname{O}(n)$ on $\mathbb{R}^n$ such that
\begin{eqnarray}\label{rotation}
g(\phi K)=\bar{\phi}\left(g(K)\right).
\end{eqnarray}
By \cite[Lemma 6.4]{LX22}, we have
\begin{eqnarray}\label{EPaffine}
{\rm \Pi}\left(\bar{\phi}\left(g(K)\right)\right)=\bar{\phi}{\rm \Pi}\left(g(K)\right).
\end{eqnarray}
Therefore,
\begin{eqnarray*}
{\rm \Pi}_{\mathbb{S}}(\phi K)=g^{-1}\left({\rm \Pi}(g(\phi K))\right)=g^{-1}\left({\rm \Pi}(\bar{\phi}g( K))\right)=g^{-1}\left(\bar{\phi}{\rm \Pi}(g( K))\right)=\phi g^{-1}\left({\rm \Pi}(g( K))\right)=\phi{\rm \Pi}_{\mathbb{S}}K,
\end{eqnarray*}
where first equality is due to (\ref{EDSPB}), the second is due to (\ref{rotation}), the third is due to (\ref{EPaffine}), the fourth is due to (\ref{rotation}) and the last equality is due to the definition of spherical projection body (\ref{EDSPB}).
\end{proof}

\subsection{Spherical projection inequality}

\begin{lem}\label{Ldecreasing}
Let $K\in\mathcal{S}_B(\mathbb{S}_+^n)$. Then
\begin{eqnarray}\label{EPiast}
\mathcal{H}^n\left({\rm \Pi}_{\mathbb{S}}^{\circ}\left(\hat{S}K\right)\right)\geq \mathcal{H}^n\left({\rm \Pi}_{\mathbb{S}}^{\circ}K\right),
\end{eqnarray}
with equality if and only if $\hat{S}K=K$.
\end{lem}
\begin{proof}
If $K\in\mathcal{S}_B(\mathbb{S}_+^n)$, then $g(K)\in\mathcal{S}_B(\mathbb{R}^n)$. For $g(K)$, by Lemma \ref{Starbodyboundary}, there is a sequence of disjoint open subsets $G_m \subset g(K)^{\prime}$, and two sequences of graph functions
$$
f_{m,j}, g_{m,j}: G_m \rightarrow \mathbb{R},\;1\leq j\leq m,
$$ satisfying (i), (ii) and (iii) in Lemma \ref{Starbodyboundary}. In particular, $\partial g(K)$ has the representation (if we neglect an $\mathcal{H}^{n-1}$-null set)
$$
\partial g(K)=\bigcup_{m=1}^{\infty} \bigcup_{j=1}^m\left\{\left(x^{\prime}, f_{m,j}\left(x^{\prime}\right)\right): x^{\prime} \in G_m\right\} \cup\left\{\left(x^{\prime}, g_{m,j}\left(x^{\prime}\right)\right): x^{\prime} \in G_m\right\} .
$$

Let $(z',t),(z',-s)\in\partial{\rm \Pi}^{\ast}\left(g(K)\right)$ and $t\neq-s$, i.e.,
\begin{eqnarray}\label{Eequ1}
h\left({\rm \Pi}\left(g(K)\right),z',t\right)=1=h\left({\rm \Pi}\left(g(K)\right),z',-s\right).
\end{eqnarray}
By the definition of projection body (\ref{Eequ1}), (\ref{Dsupportfunction}) and  (iii) in Lemma \ref{Starbodyboundary}, we have
\begin{eqnarray}\label{E1}
1&=&h\left({\rm \Pi}\left(g(K)\right),z',t\right)\\
&=&\int_{\partial g(K)}\left|(z',t)\cdot\nu^{g(K)}(x)\right|d\mathcal{H}^{n-1}(x)\nonumber\\
&=&\sum_{m=1}^{\infty}\sum_{j=1}^m\int_{\left\{(x',f_{m,j}(x')):\;x'\in G_m\right\}}\frac{\left|(z',t)\cdot\left(\nabla f_{m,j}(x'),-1\right)\right|}{\sqrt{1+\left|\nabla f_{m,j}(x')\right|^2}}d\mathcal{H}^{n-1}(x)\nonumber\\
&&+\sum_{m=1}^{\infty}\sum_{j=1}^m\int_{\left\{(x',g_{m,j}(x')):\;x'\in G_m\right\}}\frac{\left|(z',t)\cdot\left(-\nabla g_{m,j}(x'),1\right)\right|}{\sqrt{1+\left|\nabla g_{m,j}(x')\right|^2}}d\mathcal{H}^{n-1}(x)\nonumber\\
&=&\sum_{m=1}^{\infty}\sum_{j=1}^m\int_{G_m}\left|(z',t)\cdot\left(\nabla f_{m,j}(x'),-1\right)\right|dx'+\sum_{m=1}^{\infty}\sum_{j=1}^m\int_{G_m}\left|(z',t)\cdot\left(-\nabla g_{m,j}(x'),1\right)\right|dx'.\nonumber
\end{eqnarray}
Similarly, we have
\begin{eqnarray}\label{E11}
1&=&h\left({\rm \Pi}\left(g(K)\right),z',-s\right)\\
&=&\sum_{m=1}^{\infty}\sum_{j=1}^m\int_{G_m}\left|(z',-s)\cdot\left(\nabla f_{m,j}(x'),-1\right)\right|dx'+\sum_{m=1}^{\infty}\sum_{j=1}^m\int_{G_m}\left|(z',-s)\cdot\left(-\nabla g_{m,j}(x'),1\right)\right|dx'.\nonumber
\end{eqnarray}

Therefore, we have
\begin{eqnarray}\label{Ihleq1}
&&h\left({\rm \Pi}\left(\bar{S}g(K)\right),z',\frac{t+s}{2}\right)\\
&=&\int_{K'}\left|\left(z',\frac{t+s}{2}\right)\cdot\left(-r_K\nabla\overline{\varrho}(x'),1\right)\right|dx'+\int_{K'}\left|\left(z',\frac{t+s}{2}\right)\cdot\left(r_K\nabla\underline{\varrho}(x'),-1\right)\right|dx'\nonumber\\
&\leq&\int_{K'}\left|\left(z',\frac{t+s}{2}\right)\cdot\left(-\nabla\overline{\varrho}(x'),1\right)\right|dx'+\int_{K'}\left|\left(z',\frac{t+s}{2}\right)\cdot\left(\nabla\underline{\varrho}(x'),-1\right)\right|dx'\nonumber\\
&=&\int_{K'}\left|\left(z',\frac{t+s}{2}\right)\cdot\left(-\sum_{j=1}^{m}\frac{\nabla g_{m,j}(x')-\nabla f_{m,j}(x')}{2},1\right)\right|dx'\nonumber\\
&&+\int_{K'}\left|\left(z',\frac{t+s}{2}\right)\cdot\left(-\sum_{j=1}^{m}\frac{\nabla g_{m,j}(x')-\nabla f_{m,j}(x')}{2},-1\right)\right|dx'\nonumber\\
&\leq&\frac{1}{2}\sum_{m=1}^{\infty}\sum_{j=1}^{m}\int_{G_m}\left|\left(z',t\right)\cdot\left(-\nabla g_{m,j}(x'),1\right)\right|dx'+\frac{1}{2}\sum_{m=1}^{\infty}\sum_{j=1}^{m}\int_{G_m}\left|\left(z',-s\right)\cdot\left(\nabla f_{m,j}(x'),-1\right)\right|dx'\nonumber\\
&&+\frac{1}{2}\sum_{m=1}^{\infty}\sum_{j=1}^{m}\int_{G_m}\left|\left(z',t\right)\cdot\left(\nabla f_{m,j}(x'),-1\right)\right|dx'+\frac{1}{2}\sum_{m=1}^{\infty}\sum_{j=1}^{m}\int_{G_m}\left|\left(z',-s\right)\cdot\left(-\nabla g_{m,j}(x'),1\right)\right|dx'\nonumber\\
&=&\frac{1}{2}h\left({\rm \Pi}(g(K)),z',t\right)+\frac{1}{2}h\left({\rm \Pi}(g(K)),z',-s\right)=1,\nonumber
\end{eqnarray}
where the first equality is due to (\ref{DbarS}) and the same reasoning process as (\ref{E1}), the first inequality is due to $r_K\leq1$ and the monotonically increasing property of $|a+bt|+|a-bt|$ on $t>0$ for any $a,b\in\mathbb{R}$, the second equality is due to (\ref{gKandSgK}), the second inequality is due to the triangle inequalities for absolute value functions, the last two equalities are due to (\ref{E1}) and (\ref{E11}).

By (\ref{Ihleq1}) and Lemma \ref{Lprojectionbody},
\begin{eqnarray}
S{\rm \Pi}^{\ast}\left(g(K)\right)\subseteq{\rm \Pi}^{\ast}\left(\bar{S}g(K)\right).
\end{eqnarray}
By Lemma \ref{LMeasure} and the above containment relationship, we have
$$M_s\left({\rm \Pi}^{\ast}\left(g(K)\right)\right)\leq M_s\left(S{\rm \Pi}^{\ast}\left(g(K)\right)\right)\leq M_s\left({\rm \Pi}^{\ast}\left(\bar{S}g(K)\right)\right).$$
By the above inequality, the definition of spherical projection body (\ref{EDSPB}), the definition of spherical Steiner symmetrization (\ref{DSteiner}) and (\ref{EgKast}), we have
$$\mathcal{H}^n\left({\rm \Pi}_{\mathbb{S}}^{\circ}(K)\right)\leq\mathcal{H}^n\left({\rm \Pi}_{\mathbb{S}}^{\circ}(\hat{S}K)\right).$$

If $\mathcal{H}^n\left({\rm \Pi}_{\mathbb{S}}^{\circ}(K)\right)=\mathcal{H}^n\left({\rm \Pi}_{\mathbb{S}}^{\circ}(\hat{S}K)\right)$, then the equality in the first inequality of (\ref{Ihleq1}) is established.
 Thus by the arbitrariness of $z'$, $r_K=1$. By Lemma \ref{LMeasure}, $Sg(K)=g(K)$. Therefore, $\hat{S}K=K$.
\end{proof}

{\noindent}{\bf Proof of Theorem \ref{TSPB}.} By Lemma \ref{Pconverge}, there exists a sequence of directions $\{u_i\}_{i=1}^{\infty}$ such that
$\hat{S}_{u_i}\cdots\hat{S}_{u_1}K$ converges to  $K^{\star}$ in spherical Hausdorff distance.
By the continuity of spherical projection operator (see Lemma \ref{Lcontinuous}), we have
\begin{eqnarray}\label{Elimds}
\lim_{i\rightarrow\infty}d_s\left({\rm \Pi}_{\mathbb{S}}\left(\hat{S}_{u_i}\cdots\hat{S}_{u_1}K\right),{\rm \Pi}_{\mathbb{S}}(K^\star)\right)=0.
\end{eqnarray}
By Lemma \ref{Ldecreasing}, $\mathcal{H}^n\left({\rm \Pi}^{\circ}_s\left(\hat{S}_{u_i}\cdots\hat{S}_{u_1}K\right)\right)$ is increasing with respect to $i$. Thus, by (\ref{Elimds}), we have
\begin{eqnarray}\label{Egeq}
\mathcal{H}^n\left({\rm \Pi}^{\circ}_s(K)\right)\leq\mathcal{H}^n\left({\rm \Pi}^{\circ}_s(K^{\star})\right).
\end{eqnarray}

If $K\neq K^\star$, then there exists some $u_o\in\mathbb{S}^{n-1}$ such that $\hat{S}_{u_o}(K)\neq K$.
By Lemma \ref{Ldecreasing}, we have
\begin{eqnarray}\label{Ebigger}
\mathcal{H}^n\left({\rm \Pi}_{\mathbb{S}}^{\circ}(K)\right)<\mathcal{H}^n\left({\rm \Pi}_{\mathbb{S}}^{\circ}\left(\hat{S}_{u_o}(K)\right)\right).
\end{eqnarray}
By (\ref{Egeq}) and (\ref{Ebigger}), we have
$$\mathcal{H}^n\left({\rm \Pi}_{\mathbb{S}}^{\circ}(K)\right)<\mathcal{H}^n\left({\rm \Pi}^{\circ}_s(K^\star)\right).$$
Therefore, the equality in (\ref{Egeq}) holds if and only if $K=K^\star$. \qed

\subsection{Spherical projection inequality and an inequality on perimeters}
In this section, we shall prove that spherical projection inequality is stronger than an inequality on perimeters.

Let $$F_1(t):=\frac{{\pi}}{2}-\arctan t,\;\;\;t\in(0,\infty)$$
and
$$F_2(s):=\int_{0}^{s}(\sin r)^{n-1}dr,\;\;\;s\in(0,\frac{{\pi}}{2}).$$
Let $F:=F_2\circ F_1$ be the composition function of $F_1$ and $F_2$.
 It is easily to check that $F_1$ and $F_2$ are strictly convex functions, $F_2$ is strictly increasing and $F_1$ is strictly decreasing. Thus  $F$ is strictly convex and strictly decreasing.
 Let $F^{-1}$ be the inverse function of $F$, then $F^{-1}$ is also strictly convex and strictly decreasing.

\begin{thm} Let $K\in\mathcal{S}_B(\mathbb{S}_+^n)$ and $c_o=\omega_{n-1}/(n\omega_n)$. Then
\begin{align}\label{ionp}c_o\mathcal{H}^{n-1}\left(\partial g(K^\star)\right)&=F^{-1}\left(\frac{1}{n\omega_n}\mathcal{H}^n\left({\rm \Pi}^{\circ}_{\mathbb{S}}(K^\star)\right)\right)\\&\leq F^{-1}\left(\frac{1}{n\omega_n}\mathcal{H}^n\left({\rm \Pi}^{\circ}_{\mathbb{S}}(K)\right)\right)\leq c_o\mathcal{H}^{n-1}\left(\partial g(K)\right) .\nonumber\end{align}
Moreover, $\mathcal{H}^{n-1}\left(\partial g(K^\star)\right)=\mathcal{H}^{n-1}\left(\partial g(K)\right)$ if and only if $K^\star=K$.
\end{thm}

\begin{proof}

For $K\in\mathcal{S}_B(\mathbb{S}_+^n)$, by Lemma \ref{Lscoarea}, (\ref{Rhandrho}) and (\ref{Etanh}),
\begin{eqnarray}\label{Vlomeofpolar}
\mathcal{H}^n\left({\rm \Pi}^{\circ}_{\mathbb{S}}(K)\right)&=&\mathcal{H}^n\left(-{\rm \Pi}^{\circ}_{\mathbb{S}}(K)\right)\nonumber\\
&=&\int_{\mathbb{S}^{n-1}}\int_{\operatorname{conv}(e_{n+1}, v)}\sin\left(d_s\left(e_{n+1},u\right)\right)^{n-1}\chi_{-{\rm \Pi}^{\circ}_{\mathbb{S}}(K)}(u)dudv\\
&=&\int_{\mathbb{S}^{n-1}}\int_{0}^{\rho_s(-{\rm \Pi}^{\circ}_{\mathbb{S}}(K),v)}(\sin t)^{n-1}dtdv\nonumber\\
&=&\int_{\mathbb{S}^{n-1}}\int_{0}^{\frac{{\pi}}{2}-h_s({\rm \Pi}_{\mathbb{S}}(K),v)}(\sin t)^{n-1}dtdv\nonumber\\
&=&\int_{\mathbb{S}^{n-1}}\int_{0}^{\frac{{\pi}}{2}-\arctan\left(\frac{1}{2}\int_{\partial g(K)}\left|v\cdot \nu^{g(K)}(y)\right|d\mathcal{H}^{n-1}(y)\right)}(\sin t)^{n-1}dtdv.\nonumber
\end{eqnarray}
 By (\ref{Vlomeofpolar}), we have
\begin{eqnarray}\label{sigma}
\mathcal{H}^n\left({\rm \Pi}^{\circ}_{\mathbb{S}}(K)\right)=\int_{\mathbb{S}^{n-1}}F\left(\frac{1}{2}\int_{\partial g(K)}\left|v\cdot \nu^{g(K)}(y)\right|d\mathcal{H}^{n-1}(y)\right)dv.
\end{eqnarray}

 By Theorem \ref{TSPB},
 \begin{eqnarray}\label{F1}
 F^{-1}\left(\frac{1}{n\omega_n}\mathcal{H}^n\left({\rm \Pi}_{\mathbb{S}}^{\circ}(K^\star)\right)\right)\leq F^{-1}\left(\frac{1}{n\omega_n}\mathcal{H}^n\left({\rm \Pi}^{\circ}_{\mathbb{S}}(K)\right)\right).
 \end{eqnarray}
By Jensen's inequality, (\ref{sigma}) and Fubini's theorem, we have
\begin{eqnarray}\label{F2}
F^{-1}\left(\frac{1}{n\omega_n}\mathcal{H}^n\left({\rm \Pi}^{\circ}_{\mathbb{S}}(K)\right)\right)&\leq& \frac{1}{n\omega_n}\int_{\mathbb{S}^{n-1}}\left(\frac{1}{2}\int_{\partial g(K)}\left|v\cdot \nu^{g(K)}(y)\right|d\mathcal{H}^{n-1}(y)\right)dv\nonumber\\
&=&\frac{1}{2n\omega_n}\int_{\partial g(K)}\left(\int_{\mathbb{S}^{n-1}}\left|v\cdot \nu^{g(K)}(y)\right|dv\right)d\mathcal{H}^{n-1}(y)\nonumber\\
&=&\frac{\omega_{n-1}}{n\omega_n}\mathcal{H}^{n-1}\left(\partial g(K)\right).
\end{eqnarray}
Similarly, by the equality case of Jensen's inequality, (\ref{sigma}) and Fubini's theorem, we have
\begin{eqnarray}\label{F3}
F^{-1}\left(\frac{1}{n\omega_n}\mathcal{H}^n\left({\rm \Pi}^{\circ}_{\mathbb{S}}(K^\star)\right)\right)&=& \frac{1}{n\omega_n}\int_{\mathbb{S}^{n-1}}\left(\frac{1}{2}\int_{\partial g(K^\star)}\left|v\cdot \nu^{g(K^\star)}(y)\right|d\mathcal{H}^{n-1}(y)\right)dv\nonumber\\
&=&\frac{1}{2n\omega_n}\int_{\partial g(K^\star)}\left(\int_{\mathbb{S}^{n-1}}\left|v\cdot \nu^{g(K^\star)}(y)\right|dv\right)d\mathcal{H}^{n-1}(y)\nonumber\\
&=&\frac{\omega_{n-1}}{n\omega_n}\mathcal{H}^{n-1}\left(\partial g(K^\star)\right).
\end{eqnarray}
Note that the first equality of (\ref{F3}) is due to the fact that the following integral is a constant for any $v\in\mathbb{S}^{n-1}$,
$$\int_{\partial g(K^\star)}\left|v\cdot \nu^{g(K^\star)}(y)\right|d\mathcal{H}^{n-1}(y).$$
Then, the desired inequality follows from  (\ref{F1}), (\ref{F2}) and (\ref{F3}).

If $\mathcal{H}^{n-1}\left(\partial g(K^\star)\right)=\mathcal{H}^{n-1}\left(\partial g(K)\right)$, then by
strict monotonicity of $F^{-1}$ and (\ref{ionp}), we have
$\mathcal{H}^n\left({\rm \Pi}^{\circ}_{\mathbb{S}}(K^\star)\right)=\mathcal{H}^n\left({\rm \Pi}^{\circ}_{\mathbb{S}}(K)\right).$ Thus, the equality case of Theorem \ref{TSPB} gives $K^\star=K.$
\end{proof}

\section{Hyperbolic projection body and Hyperbolic projection inequality}

\subsection{$\Phi_p$ transformation of Hyperbolic space}
First, we give a transformation from Poincar\'e model $\mathbb{B}^n$ onto $\mathbb{R}^n$. Let $\Phi:\;\mathbb{B}^n:\rightarrow\mathbb{R}^n$ be a transformation given by
\begin{eqnarray}\label{y}
y=\Phi(x):=\tan\left(2\arctan\|x\|\right)\frac{x}{\|x\|}=\frac{2x}{1-\|x\|^2}.
\end{eqnarray}
Then its inverse transformation $\Phi^{-1}:\;\mathbb{R}^n\rightarrow\mathbb{B}^n$ is
\begin{eqnarray}\label{x}
x=\Phi^{-1}(y)=\frac{y}{1+\sqrt{1+\|y\|^2}}.
\end{eqnarray}
Let $\Phi_p:=\Phi\circ P$ denote the composite of $\Phi$ transformation defined in (\ref{y}) and the Poincar\'e model projection defined in Definition \ref{DPoincare}.

By (\ref{x}), we have
\begin{eqnarray}\label{dxi}
dx_i=\frac{dy_i}{1+\sqrt{1+\|y\|^2}}-\frac{y_i(y\cdot dy)}{\left(1+\sqrt{1+\|y\|^2}\right)^2\sqrt{1+\|y\|^2}}.
\end{eqnarray}
Combining (\ref{Metric}), (\ref{x}) and (\ref{dxi}), let $|dx|^2:=dx_1^2+\cdots+dx_n^2$, the metric in Poincar\'e model
\begin{eqnarray}\label{ds2y}
ds^2=4\frac{|dx|^2}{\left(1-\|x\|^2\right)^2}
&=&\sum_{i=1}^{n}\left(dy_i-\frac{y_i\left(y\cdot dy\right)}{\left(1+\sqrt{1+\|y\|^2}\right)\sqrt{1+\|y\|^2}}\right)^2\nonumber\\
&=&|dy|^2-\frac{2\sqrt{1+\|y\|^2}+2+\|y\|^2}{\left(1+\sqrt{1+\|y\|^2}\right)^2(1+\|y\|^2)}\left(y\cdot dy\right)^2\nonumber\\
&=&|dy|^2-\frac{\left(y\cdot dy\right)^2}{1+\|y\|^2}.
\end{eqnarray}
By the above equality, we have
\begin{eqnarray}\label{ds2dy2}
|dy|^2\geq ds^2\geq|dy|^2-\frac{\|y\|^2|dy|^2}{1+\|y\|^2}=\frac{|dy|^2}{1+\|y\|^2}.
\end{eqnarray}

By (\ref{x}) and  the area formula (see \cite[Theorem 3.8]{EG15}), we have the volume element
\begin{eqnarray}\label{dx}
dx:=\frac{1}{\sqrt{1+\|y\|^2}\left(1+\sqrt{1+\|y\|^2}\right)^n}dy.
\end{eqnarray}
Combining (\ref{dmun}), (\ref{x}) and (\ref{dx}) we obtain
\begin{eqnarray}\label{dmuny}
d\mu_n=\frac{1}{\sqrt{1+\|y\|^2}}dy.
\end{eqnarray}
Therefore, for $K\in\mathcal{S}_o(\mathbb{H}^n)$,
\begin{eqnarray}
\mu_n(K)=\int_{\Phi_p(K)}\frac{1}{\sqrt{1+\|y\|^2}}dy.
\end{eqnarray}
By the same proof process as that of Lemma \ref{LMeasure}, we have
\begin{eqnarray}\label{IStildeK}
\int_{S\Phi_p(K)}\frac{1}{\sqrt{1+\|y\|^2}}dy\geq\int_{\Phi_p(K)}\frac{1}{\sqrt{1+\|y\|^2}}dy,
\end{eqnarray}
with equality if and only if $S\Phi_p(K)=\Phi_p(K)$.

\begin{lem}\label{equivalent}
For $K\subset\mathbb{H}^n$, $K\in\mathcal{S}_o(\mathbb{H}^n)$  is equivalent to $\Phi_p(K)\in\mathcal{S}_o(\mathbb{R}^n)$. Moreover,
$K\in\mathcal{S}_B(\mathbb{H}^n)$  is equivalent to $\Phi_p(K)\in\mathcal{S}_B(\mathbb{R}^n)$.
\end{lem}
\begin{proof}
By (\ref{y}) and (\ref{x}), $P(K)$ is a star-shaped set with respect to origin if and only if $\Phi_p(K)$ is a star-shaped set with respect to origin in $\mathbb{R}^n$. Moreover,
by (\ref{y}) and (\ref{x}), we have
\begin{eqnarray}\label{rhoPhip}
\rho_{\Phi_p(K)}(x)=\frac{2\rho_{P(K)}(x)}{1-\|x\|^2\rho^2_{P(K)}(x)},\;\;x\in\mathbb{R}^n\backslash\{o\}
\end{eqnarray}
and
\begin{eqnarray}\label{rhop}
\rho_{P(K)}(x)=\frac{\rho_{\Phi_p(K)}(x)}{1+\sqrt{1+\|x\|^2\rho^2_{\Phi_p(K)}(x)}},\;\;x\in\mathbb{R}^n\backslash\{o\}.
\end{eqnarray}
Therefore, $\rho_{P(K)}$ is strictly positive and continuous in $\mathbb{R}^n\backslash\{o\}$ if and only if $\rho_{\Phi_p(K)}$ is strictly positive and continuous in $\mathbb{R}^n\backslash\{o\}$. Moreover, by (\ref{rhoPhip}) and (\ref{rhop}), $\rho_{P(K)}$ is locally Lipschitz continuous in $\mathbb{R}^n\backslash\{o\}$ if and only if $\rho_{\Phi_p(K)}$ is locally Lipschitz continuous in $\mathbb{R}^n\backslash\{o\}$. By \cite[Theorem 2.1]{LW23}, a Lipschitz star body (its radial function is locally Lipschitz continuous in $\mathbb{R}^n\backslash\{o\}$) is a star body with respect to a ball and vice versa.
This shows the desired conclusion.
\end{proof}

\begin{definition}\label{hpb}
For $K\in\mathcal{S}_o(\mathbb{H}^n)$, its {\it hyperbolic polar body} $K^{\circ}$ is defined by
\begin{eqnarray}\label{hyperstar}
K^{\circ}:=\Phi_p^{-1}\left(\left(\Phi_p(K)\right)^{\ast}\right).
\end{eqnarray}
\end{definition}

By the above definition and Lemma \ref{equivalent}, if $K\in\mathcal{S}_o(\mathbb{H}^n)$, then $K^{\circ}\in\mathcal{S}_o(\mathbb{H}^n)$.

\subsection{Hyperbolic Steiner symmetrization}
\begin{definition}\label{HSsymmetrization}
For $K\in\mathcal{S}_B(\mathbb{H}^n)$, its {\it hyperbolic Steiner symmetrization} $\check{S}(K)$ is defined by
\begin{eqnarray}\label{hyperSteiner}
\check{S}(K):=\Phi_p^{-1}\left(r_KS\Phi_p(K)\right),
\end{eqnarray}
where $r_K\in (0,1]$ satisfies $\mu_n\left(\check{S}(K)\right)=\mu_n(K)$.
\end{definition}

By \cite[Lemma 5.1]{LW23}, if $\tilde{K}\in\mathcal{S}_B(\mathbb{R}^n)$, then $S\tilde{K}\in\mathcal{S}_B(\mathbb{R}^n)$.
Moreover, it is clear that $rS\tilde{K}\in \mathcal{S}_B(\mathbb{R}^n)$ for $r\in (0,1]$. Thus, the hyperbolic Steiner symmetrization maintains the property of star bodies, i.e., if $K\in\mathcal{S}_B(\mathbb{H}^n)$, then $\check{S}(K)\in\mathcal{S}_B(\mathbb{H}^n)$. Moreover, by Definition \ref{HSsymmetrization}, the hyperbolic Steiner symmetrization maintains the invariance of $\mu_n$ measure.

Similarly, for compact set $K\in\mathbb{H}^n$, we define the {\it hyperbolic symmetric rearrangement} $K^{\star}$ as following
\begin{eqnarray}
K^{\star}:=\left\{v\in\mathbb{H}^n:\;ds^2(P(v),o)\leq \alpha^2,\;\;\mu_n(K)=\mu_n(B_h(\alpha))\right\}.
\end{eqnarray}

Next, we prove the convergence of hyperbolic Steiner symmetrizations.
 \begin{lem}\label{LchSs}
 For any $K\in \mathcal{S}_B(\mathbb{H}^n)$, there exists a sequence of directions $\{u_i\}_{i=1}^{\infty}\subset\mathbb{S}^{n-1}$ such that
\begin{eqnarray}
\lim_{i\rightarrow\infty}d_h\left(\check{S}_{u_i}\cdots\check{S}_{u_1}(K),K^\star\right)=0.
\end{eqnarray}
\end{lem}

\begin{proof}
Since $K\in \mathcal{S}_B(\mathbb{H}^n)$, $\Phi_p(K)\in\mathcal{S}_B(\mathbb{R}^n)$. By \cite[Theorem 3.1]{Lin21}, for the compact set $\Phi_p(K)$, there exists a sequence of directions $\{u_i\}_{i=1}^{\infty}\subset\mathbb{S}^{n-1}$ and a ball $B(r_o)\subset\mathbb{R}^n$ with the same volume as $\Phi_p(K)$
such that
\begin{eqnarray}\label{limd}
\lim_{i\rightarrow\infty}d_H\left(S_{u_i}\cdots S_{u_1}\Phi_p(K),B(r_o)\right)=0.
\end{eqnarray}
Let $r_1\in(0,1]$ satisfy
$$\Phi_p^{-1}\left(r_1S_{u_1}\Phi_p(K)\right)=\check{S}_{u_1}(K).$$
Let $r_2\in(0,1]$ satisfy
$$\Phi_p^{-1}\left(r_2S_{u_2}\Phi_p(\check{S}_{u_1}(K))\right)=\check{S}_{u_2}\check{S}_{u_1}(K).$$
Repeating the previous argument, we get a sequence of positive real numbers $\{r_i\}_{i=1}^{\infty}$ satisfying $r_i\in(0,1]$ and
\begin{eqnarray}\label{Phip-1}
\Phi_p^{-1}\left(r_iS_{u_i}\Phi_p(\check{S}_{u_{i-1}}\cdots\check{S}_{u_1}(K))\right)=\check{S}_{u_i}\check{S}_{u_{i-1}}\cdots\check{S}_{u_1}(K).
\end{eqnarray}
Let $\bar{r}:=\lim_{i\rightarrow\infty}(r_ir_{i-1}\cdots r_1)$, by (\ref{Phip-1}), (\ref{ds2dy2}) and (\ref{limd}),
\begin{eqnarray*}
&&\lim_{i\rightarrow\infty}d_h\left(\check{S}_{u_i}\cdots \check{S}_{u_1}(K),\Phi_p^{-1}\left(r_i\cdots r_1B(r_o)\right)\right)\nonumber\\
&=&\lim_{i\rightarrow\infty}d_h\left(\Phi_p^{-1}\left(r_i\cdots r_1S_{u_i}\cdots S_{u_1}\Phi_p(K)\right),\Phi_p^{-1}\left(r_i\cdots r_1B(r_o)\right)\right)\nonumber\\
&\leq&\lim_{i\rightarrow\infty}d_H\left(r_i\cdots r_1S_{u_i}\cdots S_{u_1}\Phi_p(K),r_i\cdots r_1B(r_o)\right)\nonumber\\
&=&\bar{r}\lim_{i\rightarrow\infty}d_H\left(S_{u_i}\cdots S_{u_1}\Phi_p(K),B(r_o)\right)\nonumber\\
&=&0.
\end{eqnarray*}
Since $\mu_n\left(\check{S}_{u_i}\cdots \check{S}_{u_1}(K)\right)=\mu_n(K)$ for any $i\in\mathbb{N}$, $K^\star=\Phi_p^{-1}\left(\bar{r}B(r_o)\right)$. This completes the proof.
\end{proof}

\subsection{Hyperbolic projection body}
\begin{definition}\label{hprob} For $K\in \mathcal{S}_B(\mathbb{H}^n)$, its {\it hyperbolic projection body} ${\rm \Pi}_h(K)$ is defined by
\begin{eqnarray}\label{DPihK}
{\rm \Pi}_h(K):=\Phi^{-1}_p\left({\rm \Pi} \left(\Phi_p(K)\right)\right).
\end{eqnarray}
\end{definition}
By (\ref{dmuny}),
\begin{eqnarray}\label{mundrdu}
\mu_n\left({\rm \Pi}_h(K)\right)=\int_{{\rm \Pi} \left(\Phi_p(K)\right)}\frac{1}{\sqrt{1+\|y\|^2}}dy=\int_{\mathbb{S}^{n-1}}\int_o^{\rho_{{\rm \Pi} \left(\Phi_p(K)\right)}(u)}\frac{r^{n-1}}{\sqrt{1+r^2}}drdu.
\end{eqnarray}

Next, we show that the hyperbolic projection operator $\mathcal{S}_B(\mathbb{H}^n)\rightarrow\mathcal{S}_B(\mathbb{H}^n)$ is continuous.
\begin{lem}\label{Lconhpoperator}
For $K_i\in\mathcal{S}_B(\mathbb{H}^n)$, $i=0,1,2,\dots$, if
\begin{eqnarray}\label{limKi}
\lim_{i\rightarrow\infty}d_h\left(K_i,K_o\right)=0,
\end{eqnarray}
then
\begin{eqnarray}\label{limPiKi}
\lim_{i\rightarrow\infty}d_h\left({\rm \Pi}_h(K_i),{\rm \Pi}_h(K_o)\right)=0.
\end{eqnarray}
\end{lem}
\begin{proof}
By the relation between $ds^2$ and $|dy|^2$ (see (\ref{ds2dy2})), if (\ref{limKi}) holds, then
$$\lim_{i\rightarrow\infty}d_H\left(\Phi_p(K_i),\Phi_p(K_o)\right)=0.$$
By \cite[Theorem 2.2]{LW23} and the above equality, $\Phi_p(K_i)$ converges to $\Phi_p(K_o)$ in radial distance. Thus by (\ref{performula}), their perimeters satisfy
$$\lim_{i\rightarrow\infty}\mathcal{H}^n\left(\partial \Phi_p(K_i)\right)=\mathcal{H}^n\left(\partial \Phi_p(K_o)\right).$$
By the above equality and the continuity of projection operator (see \cite[Proposition 4.1]{Lin21}), we have
\begin{eqnarray}\label{limzero}
\lim_{i\rightarrow\infty}d_H\left({\rm \Pi}\left(\Phi_p(K_i)\right),{\rm \Pi}\left(\Phi_p(K_o)\right)\right)=0.
\end{eqnarray}
By (\ref{DPihK}), (\ref{ds2dy2}) and (\ref{limzero}),
\begin{eqnarray*}
\lim_{i\rightarrow\infty}d_h\left({\rm \Pi}_h(K_i),{\rm \Pi}_h(K_o)\right)&=&\lim_{i\rightarrow\infty}d_h\left(\Phi_p^{-1}\left({\rm \Pi}\left(\Phi_p(K_i)\right)\right),\Phi_p^{-1}\left({\rm \Pi}\left(\Phi_p(K_o)\right)\right)\right)\nonumber\\
&\leq&\lim_{i\rightarrow\infty}d_H\left({\rm \Pi}\left(\Phi_p(K_i)\right),{\rm \Pi}\left(\Phi_p(K_o)\right)\right)=0.
\end{eqnarray*}
This completes the proof.
\end{proof}

The following lemma shows that the rotation invariance of the hyperbolic projection operator.
\begin{lem}
If $K\in\mathcal{S}_B(\mathbb{H}^n)$ and $\phi\in \bar{{\rm O}}(n+1)$, then
\begin{eqnarray}
{\rm \Pi}_h(\phi K)=\phi{\rm \Pi}_h(K).
\end{eqnarray}
\end{lem}
\begin{proof}
For $\phi\in\bar{{\rm O}}(n+1)$, there exists a rotation transformation $\bar{\phi}\in{\rm O}(n)$ in $\mathbb{R}^n$ such that
$$\Phi_p(\phi(K))=\bar{\phi}\left(\Phi_p(K)\right).$$
By the definition of hyperbolic projection body (\ref{DPihK}) and the affine invariance of Euclidean projection body on Lipschitz star bodies (see \cite[Lemma 6.4]{LX22}), we have
\begin{eqnarray*}
{\rm \Pi}_h(\phi K)&=&\Phi_p^{-1}\left({\rm \Pi}\left(\Phi_p\left(\phi K\right)\right)\right)=\Phi_p^{-1}\left({\rm \Pi}\left(\bar{\phi}\Phi_p\left( K\right)\right)\right)\nonumber\\
&=&\Phi_p^{-1}\left(\bar{\phi}{\rm \Pi}\left(\Phi_p\left( K\right)\right)\right)
=\phi\Phi_p^{-1}\left({\rm \Pi}\left(\Phi_p\left( K\right)\right)\right)
=\phi{\rm \Pi}_h(K).
\end{eqnarray*}
This completes the proof.
\end{proof}

\subsection{Hyperbolic projection inequality}

\begin{lem}\label{LMDPofmun}
For $K\in\mathcal{S}_B(\mathbb{H}^n)$, we have
\begin{eqnarray}\label{Increasingmun}
\mu_n\left({\rm \Pi}^{\circ}_h(K)\right)\leq \mu_n\left({\rm \Pi}^{\circ}_h(\check{S}K)\right),
\end{eqnarray}
with equality if and only if $K=\check{S}K$.
\end{lem}
\begin{proof}
By Lemma \ref{equivalent}, if $K\in\mathcal{S}_B(\mathbb{H}^n)$, then $\Phi_p(K)\in\mathcal{S}_B(\mathbb{R}^n)$.
By \cite[Theorem 7.1]{LX22},
\begin{eqnarray}\label{theorem7.1}
S{\rm \Pi}^{\ast}\left(\Phi_p(K)\right)\subset{\rm \Pi}^{\ast}\left(S\Phi_p(K)\right).
\end{eqnarray}
Therefore,
\begin{eqnarray}\label{munl}
\mu_n\left({\rm \Pi}^{\circ}_h(\check{S}K)\right)&=&\int_{\Phi_p\left({\rm \Pi}^{\circ}_h(\check{S}K)\right)}\frac{1}{\sqrt{1+\|y\|^2}}dy=\int_{{\rm \Pi}^{\ast}\left(r_KS\Phi_p(K)\right)}\frac{1}{\sqrt{1+\|y\|^2}}dy\geq\int_{{\rm \Pi}^{\ast}\left(S\Phi_p(K)\right)}\frac{1}{\sqrt{1+\|y\|^2}}dy\nonumber\\
&\geq&\int_{S{\rm \Pi}^{\ast}\left(\Phi_p(K)\right)}\frac{1}{\sqrt{1+\|y\|^2}}dy\geq\int_{{\rm \Pi}^{\ast}\left(\Phi_p(K)\right)}\frac{1}{\sqrt{1+\|y\|^2}}dy=\mu_n\left({\rm \Pi}^{\circ}_h(K)\right),
\end{eqnarray}
where the first equality is from (\ref{dmuny}), the second equality from (\ref{hyperstar}),  (\ref{DPihK}) and (\ref{hyperSteiner}), the first inequality from $r_K\in(0,1]$, the second inequality from (\ref{theorem7.1}), the third inequality from (\ref{IStildeK}) and the last equality is from (\ref{dmuny}) and (\ref{DPihK}).

If the equality in (\ref{Increasingmun}) holds, then the equality in the first inequality of (\ref{munl}) holds. This implies $r_K=1$. By the sufficient and necessary conditions of the equality holds in (\ref{IStildeK}), $S\Phi_p(K)=\Phi_p(K)$. Thus, $K=\check{S}K$.
\end{proof}

\noindent{\bf Proof of Theorem \ref{THPI}.} By the convergence of hyperbolic Steiner symmetrizations (see Lemma \ref{LchSs}), for $K\in\mathcal{S}_B(\mathbb{H}^n)$, there exists a sequence of directions $\{u_i\}_{i=1}^{\infty}\subset\mathbb{S}^{n-1}$ such that the iterative hyperbolic Steiner symmetriztions $\check{S}_{u_i}\cdots\check{S}_{u_1}K$ converge to $K^\star$ in hyperbolic Hausdorff distance. Then by the continuity of hyperbolic projection operator (see Lemma \ref{Lconhpoperator}), the sequence of the hyperbolic projection bodies ${\rm \Pi}_h\left(\check{S}_{u_i}\cdots\check{S}_{u_1}K\right)$ converge to ${\rm \Pi}_h\left(K^\star\right)$ in hyperbolic Hausdorff distance. By the  monotonically increasing property of $\mu_n\left({\rm \Pi}_{\mathbb{H}}^{\circ}(K)\right)$ with respect to hyperbolic Steiner symmetrization (see Lemma \ref{LMDPofmun}), we have
\begin{eqnarray}\label{munleqmun}
\mu_n\left({\rm \Pi}_{\mathbb{H}}^{\circ}(K)\right)\leq\mu_n\left({\rm \Pi}_{\mathbb{H}}^{\circ}(K^\star)\right).
\end{eqnarray}

If there exists $u\in\mathbb{S}^{n-1}$ such that $K\neq\check{S}_uK$, then by Lemma \ref{LMDPofmun}, $\mu_n\left({\rm \Pi}^{\circ}_h(K)\right)<\mu_n\left({\rm \Pi}^{\circ}_h(\check{S}_uK)\right)$.
By (\ref{munleqmun}) and the above inequality,
$\mu_n\left({\rm \Pi}^{\circ}_h(K)\right)<\mu_n\left({\rm \Pi}_{\mathbb{H}}^{\circ}(K^\star\right)$.
Therefore, if the equality in (\ref{munleqmun}) holds, then for any direction $u\in\mathbb{S}^{n-1}$, $K=\check{S}_uK$. By the arbitrariness of $u\in\mathbb{S}^{n-1}$, $K=K^\star$. \qed

\subsection{Hyperbolic projection inequality and isoperimetric inequality on $\mathbb{H}^n$}
In this section, we show that the hyperbolic projection inequality is stronger than the hyperbolic isoperimetric inequality in the sense of $\Phi_p$ transformation.

On the one hand,
\begin{eqnarray}\label{muleft}
\mu_n\left({\rm \Pi}_{\mathbb{H}}^{\circ}(K)\right)&=&\mu_n\left(\Phi^{-1}_p\left(\Phi^{\ast}_p\left({\rm \Pi}_h(K)\right)\right)\right)=\mu_n\left(\Phi^{-1}_p\left({\rm \Pi}^{\ast}\Phi_p(K)\right)\right)\\
&=&\int_{{\rm \Pi}^{\ast}\left(\Phi_p(K)\right)}\frac{1}{\sqrt{1+\|y\|^2}}dy=\int_{\mathbb{S}^{n-1}}\int_{0}^{h^{-1}\left({\rm \Pi}\left(\Phi_p(K)\right),u\right)}\frac{r^{n-1}}{\sqrt{1+r^2}}drdu,\nonumber
\end{eqnarray}
where the first equality is from the definition of hyperbolic polar body (\ref{hyperstar}), the second equality from the definition of hyperbolic projection body (\ref{DPihK}), the third from (\ref{dmuny}) and the last equality from (\ref{mundrdu}) and (\ref{hkrhokstar}).
Let $F_1(t):=\frac{1}{t},\;\;\;t>0$, then $F_1$ is a decreasing convex function.
Let $$F_2(s):=\int_{0}^{s}\frac{r^{n-1}}{\sqrt{1+r^2}}dr,\;\;\;s>0,$$
then $F_2$ is a increasing convex function. Thus, the composite function $F:=F_2\circ F_1$ is a decreasing convex function. Then its inverse function $F^{-1}$ is also decreasing convex function. Therefore, by (\ref{muleft}), Jensen's inequality and Fibini's theorem,
\begin{eqnarray}\label{muright}
F^{-1}\left(\frac{1}{n\omega_n}\mu_n\left({\rm \Pi}_{\mathbb{H}}^{\circ}(K)\right)\right)
&=&F^{-1}\left(\frac{1}{n\omega_n}\int_{\mathbb{S}^{n-1}}\int_{0}^{h^{-1}({\rm \Pi}\left(\Phi_p(K)\right),u)}\frac{r^{n-1}}{\sqrt{1+r^2}}drdu\right)\\
&=&F^{-1}\left(\frac{1}{n\omega_n}\int_{\mathbb{S}^{n-1}}F\left(h\left({\rm \Pi}\left(\Phi_p(K)\right),u\right)\right)du\right)\nonumber\\
&\leq&\frac{1}{n\omega_n}\int_{\mathbb{S}^{n-1}}h\left({\rm \Pi}\left(\Phi_p(K)\right),u\right)du\nonumber\\
&=&\frac{1}{2n\omega_n}\int_{\mathbb{S}^{n-1}}\left(\int_{\partial\Phi_p(K)}\left|u\cdot\nu^{\Phi_p(K)}(x)\right|d\mathcal{H}^{n-1}(x)\right)du\nonumber\\
&=&\frac{1}{2n\omega_n}\int_{\partial\Phi_p(K)}\left(\int_{\mathbb{S}^{n-1}}\left|u\cdot\nu^{\Phi_p(K)}(x)\right|du\right)d\mathcal{H}^{n-1}(x)\nonumber\\
&=&\frac{\omega_{n-1}}{n\omega_n}\mathcal{H}^{n-1}\left(\partial\Phi_p(K)\right).\nonumber
\end{eqnarray}

On the other hand, the same reasoning applies to $K^\star$, the only difference being the equality in Jensen's inequality, we have
\begin{eqnarray}\label{F-1equal}
F^{-1}\left(\frac{1}{n\omega_n}\mu_n\left({\rm \Pi}_{\mathbb{H}}^{\circ}(K^\star)\right)\right)=\frac{\omega_{n-1}}{n\omega_n}\mathcal{H}^{n-1}\left(\partial\Phi_p(K^\star\right).
\end{eqnarray}
Let $c_o:=\omega_{n-1}/(n\omega_n)$. By (\ref{IHPI}), (\ref{muright}), (\ref{F-1equal}) and the monotonicity of $F^{-1}$,
\begin{eqnarray*}
c_o\mathcal{H}^{n-1}\left(\partial\Phi_p(K^\star\right)=F^{-1}\left(\frac{1}{n\omega_n}\mu_n\left({\rm \Pi}_{\mathbb{H}}^{\circ}(K^\star)\right)\right)\leq F^{-1}\left(\frac{1}{n\omega_n}\mu_n\left({\rm \Pi}_{\mathbb{H}}^{\circ}(K)\right)\right)\leq c_o\mathcal{H}^{n-1}\left(\partial\Phi_p(K)\right).
\end{eqnarray*}
Therefore, hyperbolic projection inequality is stronger than the hyperbolic isoperimetric inequality on $\Phi_h$ transformation.
%


\bibliographystyle{plain}

\end{document}